\documentclass[11pt,letterpaper]{article}
\usepackage{setspace,graphicx}     %!  setspace is used to control linepacing
\usepackage[numbers]{natbib}
\usepackage{enumerate}  %! used for the library form, but you might find it useful too.
\pdfoutput=1
\usepackage{amsmath}
\usepackage[dvipsnames]{xcolor}
\usepackage{amssymb}
\usepackage{bbm}
\usepackage{enumerate}
\usepackage{caption}
\usepackage{amsthm}
\usepackage{relsize}
\usepackage{hyperref}
\usepackage{xspace}
\usepackage{xcolor}
\usepackage{csquotes}
\usepackage{tikz}
\usetikzlibrary{arrows}
\newtheorem{theorem}{Theorem}[section]

\newtheorem{definition}[theorem]{Definition}
\newtheorem{proposition}[theorem]{Proposition}
\newtheorem{corollary}[theorem]{Corollary}
\newtheorem{remark}[theorem]{Remark}
\newtheorem{example}[theorem]{Example}

\begin{document}

\title{A $q$-deformation of the symplectic Schur functions and the Berele insertion algorithm}
\date{ }
\author{Ioanna Nteka \footnote{E-mail address: \href{mailto:nteka.ioanna@gmail.com}{nteka.ioanna@gmail.com}}}
\maketitle

\begin{abstract}
\noindent A randomisation of the Berele insertion algorithm is proposed, where the insertion of a letter to a symplectic Young tableau leads to a distribution over the set of symplectic Young tableaux. Berele's algorithm provides a bijection between words from an alphabet and a symplectic Young tableau along with a recording oscillating tableau. The randomised version of the algorithm is achieved by introducing a parameter $0< q <1$. The classic Berele algorithm corresponds to letting the parameter $q\to 0$. The new version provides a probabilistic framework that allows to prove Littlewood-type identities for a $q$-deformation of the symplectic Schur functions. These functions correspond to multilevel extensions of the continuous $q$-Hermite polynomials. Finally, we show that when both the original and the $q$-modified insertion algorithms are applied to a random word then the shape of the symplectic Young tableau evolves as a Markov chain on the set of partitions.
\end{abstract}

\section{Introduction}
It is well known (see \cite{Robinson_1938, Scensted_1961, Stanley_2001}) that there is a one-to-one mapping, namely the Robinson-Schensted correspondence, between the set of words of length $m$ in the alphabet $\{1,...,n\}$ and the set of pairs of Young tableaux $(P,Q)$, where $P$ is a semistandard Young tableau with entries from the given alphabet and $Q$ is a standard tableau with entries $\{1,2,...,m\}$. The two tableaux have the same shape which is a partition of the length of the word. If the given word is a permutation of $\{1,...,n\}$, then $P,Q$ are both standard Young tableaux with the same shape. Hence the Robinson-Schensted algorithm provides a bijection between the symmetric group $\mathfrak{S}_{n}$ and pairs of standard Young tableaux with the same shape, which implies the fundamental identity
\[\sum_{\lambda \in \Lambda_{n}}|\{P: \text{standard Young tableau of shape } \lambda\}|^2=n!=|\mathfrak{S}_{n}|\]
where $\Lambda_{n}$ denotes the set of partitions of length at most $n$.\\

Berele in \cite{Berele_1986} proved that similarly to the Robinson-Schensted algorithm, there exists a one-to-one correspondence between words of length $m$ from the alphabet $\{1,\bar{1},...,n, \bar{n}\}$  and the sets of pairs $(P,f)$, where $P$ is a symplectic tableau and $f = (f^0,...,f^m)$ is a recording sequence of oscillating diagrams such that the shape of $P$ is given by $f^m$. The Berele correspondence can be used to prove Littlewood-type identities associated with the symplectic group. For example, Berele's scheme gives a combinatorial proof for the following identity (\cite{Sundaram_1990b})
\begin{equation}
\label{eq:Schur_decomposition}
(a_{1}+a_{1}^{-1}+...+a_{n}+a_{n}^{-1})^m = \sum_{\lambda \in \Lambda_{n}} Q^{\lambda}_{m}(n)Sp_{\lambda}^{(n)}(a_{1},...,a_{n})
\end{equation}
where we denote by $Sp_{\lambda}^{(n)}$ the symplectic Schur function, namely the character of an irreducible finite-dimensional representation of the symplectic group $Sp(2n)$, parametrized by $\lambda$ and we write $Q^{\lambda}_{m}(n)$ for the number of oscillating sequences of diagrams $ f^0=\emptyset ,f^1,...,f^m=\lambda$ with length at most $n$. In the thesis of Sundaram \cite{Sundaram_1986} more results for the symplectic Schur functions are proved using Berele's bijection.

In this paper we will propose a randomisation of Berele's algorithm introducing a tuning parameter $q\in (0,1)$. This will enable us to derive a generalisation of identity \eqref{eq:Schur_decomposition} where the symplectic Schur functions are replaced by a multivariate extension of the continuous $q$-Hermite polynomials (for an extensive review on the properties of the continuous $q$-Hermite polynomials and their connetctions with other basic hypergeometric polynomials we refer the reader to \cite{Koekoek_Lesky_Swarttouw_2010}), which we denote by $P_{\lambda}^{(n)}(a_{1},...,a_{n};q)$. In Corollary \ref{q-Littlewood(corollary)} we present the following identity
\begin{equation*}
(a_{1}+a_{1}^{-1}+...+a_{n}+a_{n}^{-1})^m = \sum_{\lambda\in \Lambda_{n}} Q^{\lambda}_{m}(n;q)P_{\lambda}^{(n)}(a_{1},...,a_{n};q).
\end{equation*}
where  the function $Q^{\lambda}_{m}(n;q)$ is given as the sum over sequences of diagrams $f^0=\emptyset,f^1,...,f^m=\lambda$ with length at most $n$ where each such sequence has a weight that depends on the parameter $q$. As $q\to 0$, the quantity $Q^{\lambda}_{m}(n;q)$ converges to $Q^{\lambda}_{m}(n)$. Later we will also see that the symplectic Schur function $Sp_{\lambda}^{(n)}(a_{1},...,a_{n})$ corresponds to the polynomial $P_{\lambda}^{(n)}(a_{1},...,a_{n};q)$ with $q = 0$, recovering identity \eqref{eq:Schur_decomposition}. We remark that although the left hand sides of identity \eqref{eq:Schur_decomposition} and the $q$-deformed identity are the same and their right hand sides consist of the same number of terms, these terms do not coincide with each other. 

A randomisation of the column insertion version of the Robinson-Schensted algorithm was previously proposed by O'Connell-Pei in \cite{O'Connell_Pei_2013}, where the insertion process is formulated using insertion paths. A randomisation of the Robinson-Schensted algorithm in the framework of Gelfand-Tsetlin patterns can be found in the work of Borodin-Petrov \cite{Borodin_Petrov_2013} and Matveev-Petrov \cite{Matveev_Petrov_2016}. In our work, we follow the latter approach.

If the Robinson-Schensted algorithm is applied to a random word from the alphabet $\{1,...,n\}$ then O'Connell proved in \cite{O'Connell_2003b} that the shape of the Young tableau evolves as a Markov chain on the set of partitions with length at most $n$. The transition probabilities are given in terms of the Schur functions. Similarly, the $q$-weighted version by O'Connell-Pei leads to a random walk on $\Lambda_{n}$ with transition probabilities involving $q$-Whittaker functions instead.  In this paper we will show that if we insert a random word from the alphabet $\{1,\bar{1},...,n,\bar{n}\}$ according to the Berele insertion algorithm or its $q$-deformation proposed, then the shape of the symplectic Young tableau evolves as a Markov chain on $\Lambda_{n}$ with transition probabilities given in terms of the symplectic Schur functions and the functions $P^{(n)}(\cdot\, ;q)$ respectively.

The outline of the paper is the following. In section 2 we introduce all the relevant notation. In section 3 we present the insertion algorithm proposed by Berele and study the evolution of the shape of a Young tableau when we insert randomly chosen letters. In section 4 we describe the algorithm using Gelfand-Tsetlin patterns instead of Young tableau, which helps us to $q$-modify the insertion process in section 5. In section 6 we state the main results obtained using an intertwining argument. In section 7 we study the shape of the symplectic Young tableau when the $q$-deformation of the Berele insertion algorithm is applied to a random word. Finally, the proof of the main intertwining identity can be found in section 8.

\section{Partitions and Young diagrams}
\label{partitions}
A \emph{partition} is a sequence $\lambda = (\lambda_{1},\lambda_{2}, ...)$ of integers satisfying $\lambda_{1}\geq \lambda_{2}\geq ... \geq 0$ with finitely many non-zero terms. We call each term $\lambda_{i}$ a \emph{part} of $\lambda$ and the number of non-zero parts the \emph{length of the partition} $\lambda$, denoted by $l(\lambda)$. Moreover we call $|\lambda|:=\sum_{i\geq 1}\lambda_{i}$, the \emph{weight of the partition} $\lambda$. If $|\lambda|=k$, then we say that $\lambda$ partitions $k$ and we write $\lambda \vdash k$. The set of partitions of length at most $n$ is denoted by $\Lambda_{n}$.

We define the natural ordering on the space of partitions called the \emph{dominance order}. For two partitions $\lambda, \mu$ we write $\lambda \geq \mu$ if and only if
\[\lambda_{1}+...+\lambda_{i}\geq \mu_{1}+...+\mu_{i}, \text{ for all }i\geq 1.\]

We also define the notion of interlacing of partitions. Let $\lambda$ and $\mu$ be two partitions with $\lambda\geq \mu$. The partitions $\lambda$, $\mu$ are said to be \emph{interlaced}, and we write $\mu\preceq \lambda$ if and only if
\[\lambda_{1}\geq \mu_{1}\geq \lambda_{2}\geq \mu_{2}\geq ...\quad  .\]

A partition $\lambda$ can be represented in the plane by an arrangement of boxes called a \emph{Young diagram}. This arrangement is top and left justified with $\lambda_{i}$ boxes in the $i$-th row. We then say that the shape of the diagram, denoted by $sh$, is $\lambda$. For example the partition $\lambda = (3,2,2,1)$ can be represented as in figure \ref{fig:diagram}.

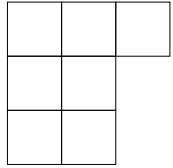
\begin{figure}[h]
\begin{center}
\scalebox{0.9}{
\begin{tikzpicture}
\draw [black] (0,0) rectangle (0.8,0.8);
\draw [black] (0.8,0) rectangle (1.6,0.8);
\draw [black] (1.6,0) rectangle (2.4,0.8);
\draw [black] (0,-0.8) rectangle (0.8,0);
\draw [black] (0.8,-0.8) rectangle (1.6,0);
\draw [black] (0,-1.6) rectangle (0.8,-0.8);
\draw [black] (0.8,-1.6) rectangle (1.6,-0.8);
%\draw [black] (0,-2.4) rectangle (0.8,-1.6);
\end{tikzpicture}
}
\caption{A Young diagram of shape $(3,2,2)$.}
\label{fig:diagram}
\end{center}
\end{figure}

Given two partitions/diagrams $\lambda$ and $\mu$ such that $\mu \subset \lambda$ (as a set of boxes), we call the set difference $\lambda \setminus \mu$ a \emph{skew Young diagram}. A skew Young diagram $\lambda \setminus \mu$ is a horizontal strip if in each column it has at most one box. 

Let us now consider a filling of a Young diagram, namely the symplectic Young tableau, with entries from an alphabet $[n,\bar{n}]:=\{1<\bar{1}<...<n< \bar{n}\}$. For the rest of this paper, we will refer to the relative position of a letter $l$ in the alphabet $[n,\bar{n}]$ as the \emph{order} of the letter and write $order(l)$. 
\begin{definition}[\cite{King_1971}]
\label{symp_tableau_def}
A \emph{symplectic Young tableau} is a filling of a Young diagram with entries from the alphabet $[n,\bar{n}]$ satisfying the following conditions.\\
\textbf{(S1)} The entries are weakly increasing along the rows;\\
\textbf{(S2)} The entries are strictly increasing down the columns;\\
\textbf{(S3)} No entry $<i$ occurs in row $i$, for any i.\\
We denote the set of symplectic tableaux with entries from $[n,\bar{n}]$ by $\mathcal{T}_{[n,\bar{n}]}$.
\end{definition}

\begin{example}The filling of the diagram of shape $(4,3)$
%\newpage
\begin{figure*}[h]
\begin{center}
\scalebox{0.9}{
\begin{tikzpicture}
\draw [black] (0,0) rectangle (0.8,0.8);
\node at (0.4,0.4) {1};
\draw [black] (0.8,0) rectangle (1.6,0.8);
\node at (1.2,0.4) {$\bar{1}$};
\draw [black] (1.6,0) rectangle (2.4,0.8);
\node at (2,0.4) {$\bar{1}$};
\draw [black] (2.4,0) rectangle (3.2,0.8);
\node at (2.8,0.4) {$\bar{2}$};
\draw [black] (0,-0.8) rectangle (0.8,0);
\node at (0.4,-0.4) {2};
\draw [black] (0.8,-0.8) rectangle (1.6,0);
\node at (1.2,-0.4) {2};
\draw [black] (1.6,-0.8) rectangle (2.4,0);
\node at (2,-0.4) {$\bar{2}$};
\end{tikzpicture}
}
\end{center} 
\end{figure*}
\noindent is a symplectic Young tableau, whereas the filling

\begin{figure*}[h]
\begin{center}
\scalebox{0.9}{
\begin{tikzpicture}
\draw [black] (0,0) rectangle (0.8,0.8);
\node at (0.4,0.4) {1};
\draw [black] (0.8,0) rectangle (1.6,0.8);
\node at (1.2,0.4) {$\bar{1}$};
\draw [black] (1.6,0) rectangle (2.4,0.8);
\node at (2,0.4) {$\bar{1}$};
\draw [black] (2.4,0) rectangle (3.2,0.8);
\node at (2.8,0.4) {$\bar{2}$};
\draw [black] (0,-0.8) rectangle (0.8,0);
\node at (0.4,-0.4) {$\bar{1}$};
\draw [black] (0.8,-0.8) rectangle (1.6,0);
\node at (1.2,-0.4) {2};
\draw [black] (1.6,-0.8) rectangle (2.4,0);
\node at (2,-0.4) {$\bar{2}$};
\end{tikzpicture}
}
\end{center} 
\end{figure*}
\noindent is not since there exists an entry $\bar{1}$ at the second row violating condition $\textbf{(S3)}$.
\end{example}

We will now define an oscillating tableau, which plays the role of a recording tableau.
\begin{definition}
An \emph{oscillating tableau} of length $m$ is a sequence of partitions $f=(f^0,...,f^m)$, with $f^0=\emptyset$ such that any two consecutive partitions differ by exactly one box, i.e. for $1\leq i \leq m$, it holds that $f^i \setminus f^{i-1}=(1)$ or $f^{i-1}\setminus f^i = (1)$. We say that an oscillating tableau $f$ of length $m$ has \emph{shape} $\lambda$ if $f^m = \lambda$.
\end{definition}

We will denote the set of oscillating tableaux of length $m$ by $\mathcal{O}_{m}$ and the set of oscillating tableau $f\in \mathcal{O}_{m}$, which moreover satisfy the property that for each $1\leq i \leq m$, $f^i \in \Lambda_{n}$ by $\mathcal{O}_{m}(n)$. Finally, for $\lambda \in \Lambda_{n}$ we write $Q_{m}^{\lambda}(n)$ for the number of oscillating tableau $f\in \mathcal{O}_{m}(n)$ of shape $\lambda$.

\section{The Berele insertion algorithm}
We recall that the Robinson-Schensted algorithm has two versions; the row insertion and the column insertion. The Berele algorithm follows the row insertion version. Before we describe the insertion algorithm we need to introduce a sliding algorithm called the \textit{jeu de taquin} algorithm.
\begin{definition}
A punctured tableau of shape $z$ is a Young diagram of shape $z$ in which every box except one is filled. We will refer to this special box as the empty box. 
\end{definition}
\begin{definition}[jeu de taquin]
Let $T$ be a punctured tableau with $(\alpha, \beta)$-entry, where $\alpha$ denotes the row and $\beta$ the column, $t_{\alpha \beta}$ and with empty box in position $(i,j)$. We consider the transformation $jdt: T \to jdt(T)$ defined as follows
\begin{itemize}
\item if $T$ is an ordinary tableau then $jdt(T)=T$;
\item while $T$ is a punctured tableau
\begin{equation*}
T \rightarrow \left\{ \begin{array}{ll}
 T\text{ switching the empty box and } t_{i,j+1} & \text{, if } t_{i,j+1}<t_{i+1,j}\\
 T\text{ switching the empty box and }t_{i+1,j} & \text{, if }t_{i,j+1}\geq t_{i+1,j}.
\end{array} \right.
\end{equation*}
Here we will use the convention that if the empty box has only one right/down neighbouring box $(\alpha, \beta)$, then 
\[T \rightarrow T \text{ switching the empty box and }t_{\alpha, \beta}.\]
\end{itemize}
\end{definition}

\begin{example} \hspace{5pt }
\begin{figure}[h]
\begin{center}
\scalebox{0.9}{
\begin{tikzpicture}
\node at (-0.7,0) {$T =$ };
\draw [black] (-0.2,0) rectangle (0.6,0.8);
\draw [black] (0.6,0) rectangle (1.4,0.8);
\draw [black] (1.4,0) rectangle (2.2,0.8);
\draw [black] (-0.2,-0.8) rectangle (0.6,0);
\draw [black] (0.6,-0.8) rectangle (1.4,0);
\node at (1,0.4) {1};
\node at (1.8,0.4) {2};
\node at (0.2,-0.4) {2};
\node at (1,-0.4) {2};

\draw [thick] [->] (2.4,0) -- (3.9,0) node[above] {\footnotesize{$\, t_{12}<t_{21}\qquad \qquad \quad$}};

\draw [black] (4.1,0) rectangle (4.9,0.8);
\draw [black] (4.9,0) rectangle (5.7,0.8);
\draw [black] (5.7,0) rectangle (6.5,0.8);
\draw [black] (4.1,-0.8) rectangle (4.9,0);
\draw [black] (4.9,-0.8) rectangle (5.7,0);
\node at (4.5,0.4) {1};
\node at (6.1,0.4) {2};
\node at (4.5,-0.4) {2};
\node at (5.3,-0.4) {2};

\draw [thick] [->] (6.8,0) -- (8.3,0) node[above] {\footnotesize{$\, t_{13}=t_{22}\qquad \qquad\quad $} };

\draw [black] (8.5,0) rectangle (9.3,0.8);
\draw [black] (9.3,0) rectangle (10.1,0.8);
\draw [black] (10.1,0) rectangle (10.9,0.8);
\draw [black] (8.5,-0.8) rectangle (9.3,0);
\node at (8.9,0.4) {1};
\node at (10.5,0.4) {2};
\node at (8.9,-0.4) {2};
\node at (9.7,0.4) {2};

\node at (12,0) {$=jdt(T)$};
\end{tikzpicture}
}
\end{center} 
\end{figure}
\end{example}

We are now ready to describe the Berele insertion algorithm. To insert a letter $i$ from the alphabet $[n, \bar{n}]$  to a symplectic tableau $P$, we begin by trying to place the letter at the end of the first row. If the result is a symplectic tableau we are done. Otherwise, the smallest entry which is larger than $i$ is bumped and we proceed by inserting the bumped letter to the second row and so on. Suppose now that at some instance of the insertion process condition \textbf{(S3)} in Definition \ref{symp_tableau_def} is violated. This means that we tried to insert a letter $l$ to the $l$-th row of the tableau, for some $1\leq l \leq n$, and bumped an $\bar{l}$ letter. Since we cannot insert $\bar{l}$ to the $(l+1)$-th row we erase both $l$ and $\bar{l}$, leaving the box formerly occupied by the $\bar{l}$ as an empty box. We then apply jeu de taquin algorithm.

\begin{example}
Insert $\bar{1}$ to 
\newpage
\begin{figure}[h]
\begin{center}
\scalebox{0.9}{
\begin{tikzpicture}
\node at (-0.5,-0.4) {P = };
\draw [black] (0,0) rectangle (0.8,0.8);
\draw [black] (0.8,0) rectangle (1.6,0.8);
\draw [black] (1.6,0) rectangle (2.4,0.8);
\draw [black] (2.4,0) rectangle (3.2,0.8);
\draw [black] (0,-0.8) rectangle (0.8,0);
\draw [black] (0.8,-0.8) rectangle (1.6,0);
\draw [black] (1.6,-0.8) rectangle (2.4,0);
\draw [black] (0,-1.6) rectangle (0.8,-0.8);
\draw [black] (0.8,-1.6) rectangle (1.6,-0.8);
\node at (0.4,0.4) {1};
\node at (1.2,0.4) {1};
\node at (2,0.4) {2};
\node at (2.8,0.4) {$\bar{2}$};
\node at (0.4,-0.4) {2};
\node at (1.2,-0.4) {$\bar{2}$};
\node at (2,-0.4) {3};
\node at (0.4,-1.2) {3};
\node at (1.2,-1.2) {$\bar{3}$};
\node at (3.6,0.4) {+$\bar{1}$};

\draw [thick] [->] (4,-0.4) -- (5.5,-0.4) node[above] {bump 2 \hspace{40pt} };

\draw [black] (6,0) rectangle (6.8,0.8);
\draw [black] (6.8,0) rectangle (7.6,0.8);
\draw [black] (7.6,0) rectangle (8.4,0.8);
\draw [black] (8.4,0) rectangle (9.2,0.8);
\draw [black] (6,-0.8) rectangle (6.8,0);
\draw [black] (6.8,-0.8) rectangle (7.6,0);
\draw [black] (7.6,-0.8) rectangle (8.4,0);
\draw [black] (6,-1.6) rectangle (6.8,-0.8);
\draw [black] (6.8,-1.6) rectangle (7.6,-0.8);
\node at (6.4,0.4) {1};
\node at (7.2,0.4) {1};
\node at (8,0.4) {$\mathbf{\bar{1}}$};
\node at (8.8,0.4) {$\bar{2}$};
\node at (6.4,-0.4) {2};
\node at (7.2,-0.4) {$\bar{2}$};
\node at (8,-0.4) {3};
\node at (6.4,-1.2) {3};
\node at (7.2,-1.2) {$\bar{3}$};
\node at (9.6,-0.4) {+2};

\draw [thick] [->] (10,-0.4) -- (11.5,-0.4) node[above] {bump $\bar{2}$ \hspace{40pt} };

\draw [black] (0,-3) rectangle (0.8,-2.2);
\draw [black] (0.8,-3) rectangle (1.6,-2.2);
\draw [black] (1.6,-3) rectangle (2.4,-2.2);
\draw [black] (2.4,-3) rectangle (3.2,-2.2);
\draw [black] (0,-3.8) rectangle (0.8,-3);
\draw [black] (0.8,-3.8) rectangle (1.6,-3);
\draw [black] (1.6,-3.8) rectangle (2.4,-3);
\draw [black] (0,-4.6) rectangle (0.8,-3.8);
\draw [black] (0.8,-4.6) rectangle (1.6,-3.8);
\node at (0.4,-2.6) {1};
\node at (1.2,-2.6) {1};
\node at (2,-2.6) {$\bar{1}$};
\node at (2.8,-2.6) {$\bar{2}$};
\node at (0.4,-3.4) {2};
\node at (1.2,-3.4) {$\mathbf{2}$};
\node at (2,-3.4) {3};
\node at (0.4,-4.2) {3};
\node at (1.2,-4.2) {$\bar{3}$};
\node at (3.6,-4.2) {+$\bar{2}$};

\draw [thick] [->] (4,-3.4) -- (5.5,-3.4) node[above] {cancel 2,$\bar{2}$ \hspace{40pt} };

\draw [black] (6,-3) rectangle (6.8,-2.2);
\draw [black] (6.8,-3) rectangle (7.6,-2.2);
\draw [black] (7.6,-3) rectangle (8.4,-2.2);
\draw [black] (8.4,-3) rectangle (9.2,-2.2);
\draw [black] (6,-3.8) rectangle (6.8,-3);
\draw [black] (6.8,-3.8) rectangle (7.6,-3);
\draw [black] (7.6,-3.8) rectangle (8.4,-3);
\draw [black] (6,-4.6) rectangle (6.8,-3.8);
\draw [black] (6.8,-4.6) rectangle (7.6,-3.8);
\node at (6.4,-2.6) {1};
\node at (7.2,-2.6) {1};
\node at (8,-2.6) {$\bar{1}$};
\node at (8.8,-2.6) {$\bar{2}$};
\node at (6.4,-3.4) {2};
\node at (8,-3.4) {3};
\node at (6.4,-4.2) {3};
\node at (7.2,-4.2) {$\bar{3}$};

\draw [thick] [->] (10,-3.4) -- (11.5,-3.4) node[above] {jeu de taquin \hspace{40pt} };

\draw [black] (0,-6) rectangle (0.8,-5.2);
\draw [black] (0.8,-6) rectangle (1.6,-5.2);
\draw [black] (1.6,-6) rectangle (2.4,-5.2);
\draw [black] (2.4,-6) rectangle (3.2,-5.2);
\draw [black] (0,-6.8) rectangle (0.8,-6);
\draw [black] (0.8,-6.8) rectangle (1.6,-6);
\draw [black] (0,-7.6) rectangle (0.8,-6.8);
\draw [black] (0.8,-7.6) rectangle (1.6,-6.8);
\node at (0.4,-5.6) {1};
\node at (1.2,-5.6) {1};
\node at (2,-5.6) {$\bar{1}$};
\node at (2.8,-5.6) {$\bar{2}$};
\node at (0.4,-6.4) {2};
\node at (1.2,-6.4) {3};
\node at (0.4,-7.2) {3};
\node at (1.2,-7.2) {$\bar{3}$};
\node at (4.5,-6.4) { = (P $\stackrel{\mathcal{B}}{\leftarrow} \bar{1}$)};
\end{tikzpicture}
}
\end{center} 
\end{figure}

\end{example}
The insertion algorithm can be applied to a word $w=w_{1},...,w_{m}$, with $w_{i}\in [n,\bar{n}]$, starting with an empty tableau. The output, denoted by $\mathcal{B}(w)$, is a symplectic Young tableau along with an oscillating tableau $f = (f^{0},...,f^{m})$ that records the shapes of the symplectic tableau for all the intermediate steps. More specifically, if we denote by $P(i)$ the tableau after the insertion of the $i$-th letter then $shP(i)=f^{i}$ for $i=1,..,m$. 

\begin{example}
Applying the Berele algorithm to the word $w=\bar{3}2\bar{1}\bar{3}121$ yields the pair $\mathcal{B}(w)=(P,f=(f^{0},...,f^{7}))$, where
\begin{figure}[h]
\begin{center}
\scalebox{0.9}{
\begin{tikzpicture}
\node at (-0.4,-0.4) {$P(i):$};
\node at (-0.4,-3) {$f^i:$};

\node at (0.4,-0.4) {$\emptyset$};
\node at (0.4,-3) {$\emptyset$};

\draw [black] (1.3,0) rectangle (2.1,-0.8);
\node at (1.7,-0.4) {$\bar{3}$};
\node at (1.7,-3) {$(1)$};

\draw [black] (2.6,0) rectangle (3.4,-0.8);
\node at (3,-0.4) {2};
\draw [black] (2.6,-0.8) rectangle (3.4,-1.6);
\node at (3,-1.2) {$\bar{3}$};
\node at (3,-3) {$(1,1)$};

\draw [black] (3.9,0) rectangle (4.7,-0.8);
\node at (4.3,-0.4) {$\bar{1}$};
\draw [black] (3.9,-0.8) rectangle (4.7,-1.6);
\node at (4.3,-1.2) {2};
\draw [black] (3.9,-1.6) rectangle (4.7,-2.4);
\node at (4.3,-2) {$\bar{3}$};
\node at (4.3,-3) {$(1,1,1)$};

\draw [black] (5.2,0) rectangle (6,-0.8);
\node at (5.6,-0.4) {$\bar{1}$};
\draw [black] (5.2,-0.8) rectangle (6,-1.6);
\node at (5.6,-1.2) {2};
\draw [black] (5.2,-1.6) rectangle (6,-2.4);
\node at (5.6,-2) {$\bar{3}$};
\draw [black] (6,0) rectangle (6.8,-0.8);
\node at (6.4,-0.4) {$\bar{3}$};
\node at (6.2,-3) {$(2,1,1)$};

\draw [black] (7.3,0) rectangle (8.1,-0.8);
\node at (7.7,-0.4) {$2$};
\draw [black] (7.3,-0.8) rectangle (8.1,-1.6);
\node at (7.7,-1.2) {$\bar{3}$};
\draw [black] (8.1,0) rectangle (8.9,-0.8);
\node at (8.5,-0.4) {$\bar{3}$};
\node at (8.1,-3) {$(2,1)$};

\draw [black] (9.4,0) rectangle (10.2,-0.8);
\node at (9.8,-0.4) {$2$};
\draw [black] (9.4,-0.8) rectangle (10.2,-1.6);
\node at (9.8,-1.2) {$\bar{3}$};
\draw [black] (10.2,0) rectangle (11,-0.8);
\node at (10.6,-0.4) {$2$};
\draw [black] (10.2,-0.8) rectangle (11,-1.6);
\node at (10.6,-1.2) {$\bar{3}$};
\node at (10.2,-3) {$(2,2)$};

\draw [black] (11.5,0) rectangle (12.3,-0.8);
\node at (11.9,-0.4) {$1$};
\draw [black] (11.5,-0.8) rectangle (12.3,-1.6);
\node at (11.9,-1.2) {$2$};
\draw [black] (11.5,-1.6) rectangle (12.3,-2.4);
\node at (11.9,-2) {$\bar{3}$};
\draw [black] (12.3,0) rectangle (13.1,-0.8);
\node at (12.7,-0.4) {$2$};
\draw [black] (12.3,-0.8) rectangle (13.1,-1.6);
\node at (12.7,-1.2) {$\bar{3}$};
\node at (12.4,-3) {$(2,2,1)$ .};
\end{tikzpicture}
}
\end{center} 
\end{figure}
\end{example}

Analogously to the Robinson-Schensted correspondence, Berele proved the following result.
\begin{theorem}[\cite{Berele_1986}]
$\mathcal{B}$ is a bijection between words $w_{1},...,w_{m}$ in the alphabet $[n,\bar{n}] = \{1,\bar{1},...,n ,\bar{n}\}$ and pairs $(P,f)\in \mathcal{T}_{[n,\bar{n}]}\times \mathcal{O}_{m}(n)$ in which $P$ is a symplectic tableau and $f$ is an oscillating tableau with shape $f^m = sh P$.
\end{theorem}

O'Connell observed, in \cite{O'Connell_2003b}, that if the Robinson-Schensted algorithm is applied to a random word, then the shape of the output tableau evolves as a Markov chain on the set of partitions with transition kernel given in terms of Schur functions. \\

We will devote the rest of this section to the study of the evolution of the shape of the symplectic tableau when we Berele-insert a random word from the alphabet $[n,\bar{n}]$. We first need to introduce some notation.\\

Fix $n\geq 1$ and $a=(a_{1},...,a_{n})\in \mathbb{R}^n_{>0}$. If $w=w_{1},...,w_{m}$ is a word from $[n,\bar{n}]$, we write
\[a^w = a_{w_{1}}...a_{w_{m}}\]
where we use the convention that $a_{\bar{l}}=a_{l}^{-1}$, for $1\leq l \leq n$.\\

\noindent For a symplectic Young tableau $P$ with entries from $[n,\bar{n}]$, we write
\[a^P = \prod_{l=1}^n a_{l}^{\#\{l \text{'s  in }P\}-\#\{\bar{l}\text{'s  in }P\}}.\]

We can now give a combinatorial definition for the symplectic Schur functions.
\begin{definition}(\cite{King_1971})
Let $\lambda\in \Lambda_{n}$ be a partition with length at most $n$. The \emph{symplectic Schur function} parametrized by $\lambda$ is given by
\[Sp_{\lambda}^{(n)}(a) = \sum_{\substack{P\in \mathcal{T}_{[n,\bar{n}]}:\\ shP=\lambda}}a^P.\]
\label{Schur_combinatorial}
\end{definition}
The symplectic Schur function satisfies the Pieri rule, i.e for every $\lambda\in \Lambda_{n}$ and $a\in \mathbb{R}^n_{>0}$ the following identity holds (\cite{Sundaram_1990c})
\begin{equation}
\label{eq:Pieri_Schur}
Sp_{\lambda}^{(n)}(a)\sum_{i=1}^n(a_{i}+a_{i}^{-1})=\sum_{l=1}^n\big( Sp_{\lambda+e_{l}}^{(n)}(a)\mathbbm{1}_{\lambda+e_{l}\in \Lambda_{n}}+Sp_{\lambda-e_{l}}^{(n)}(a)\mathbbm{1}_{\lambda-e_{l}\in \Lambda_{n}}\big).
\end{equation}\\

Let us now consider a sequence of independent random variables $\{w_{k}$, $k\geq 1\}$ taking values in $[n,\bar{n}]$ with common distribution given, for $1\leq l \leq n$, by
\begin{equation}
\label{eq:letter_distribution}
\begin{array}{ll}
\rho(l):=\mathbb{P}[w_{k}=l] = \dfrac{a_{l}}{\sum_{i=1}^n(a_{i}+a_{i}^{-1})}\\\rho(\bar{l}):=\mathbb{P}[w_{k}=\bar{l}] = \dfrac{a_{l}^{-1}}{\sum_{i=1}^n(a_{i}+a_{i}^{-1})}
\end{array} .
\end{equation}\\

Let $(P(m),f(m))\in \mathcal{T}_{[n,\bar{n}]}\times \mathcal{O}_{m}(n)$ be the pair of tableau obtained when we apply the Berele algorithm to the random word $w_{1},...,w_{m}$ from $[n,\bar{n}]$ with distribution $\rho$. Then, due to the fact that the Berele insertion algorithm is a bijection between words and pairs of tableaux, we conclude that for $(P,f= (f^0,...,f^m))\in \mathcal{T}_{[n,\bar{n}]}\times \mathcal{O}_{m}(n)$
\begin{equation}
\mathbb{P}[(P(m),f(m))=(P,f)] = \dfrac{a^P}{(\sum_{i=1}^n(a_{i}+a_{i}^{-1}))^m}\mathbbm{1}_{shP=f^m}.
\label{eq:probability_(P,f)_Berele}
\end{equation}
Summing \eqref{eq:probability_(P,f)_Berele}
over $(P,f= (f^0,...,f^m))\in \mathcal{T}_{[n,\bar{n}]}\times \mathcal{O}_{m}(n)$ such that $shP = f^m = \lambda$ we conclude that
\[\mathbb{P}[shP(m) = \lambda] = \dfrac{1}{(\sum_{i=1}^n(a_{i}+a_{i}^{-1}))^m}Sp_{\lambda}^{(n)}(a)Q_{m}^{\lambda}(n)\]
where we recall that $Q_{m}^{\lambda}(n)$ is the number of oscillating tableaux $f \in \mathcal{O}_{m}(n)$ with shape $\lambda$. Using the fact that
\[\sum_{\lambda \in \Lambda_{n}}\mathbb{P}[shP(m)=\lambda] = 1\] 
we may recover the following well-known character identity (\cite{Sundaram_1990b})
\[\sum_{\lambda\in \Lambda_{n}}Sp_{\lambda}^{(n)}(a)Q_{m}^{\lambda}(n) = \Big(\sum_{i=1}^n(a_{i}+a_{i}^{-1})\Big)^m.\]
If on the other hand we sum \eqref{eq:probability_(P,f)_Berele} over $P \in \mathcal{T}_{[n,\bar{n}]}$ we obtain a distribution for the whole evolution of shapes of $P$, i.e. we conclude that for $f = (f^0,f^1,...,f^m)\in \mathcal{O}_{m}(n)$
\[\mathbb{P}[shP(1)=f^1,...,shP(m)=f^m] = \dfrac{1}{(\sum_{i=1}^n(a_{i}+a_{i}^{-1}))^m}Sp_{f^m}^{(n)}(a).\]
Note that, by the definition of conditional probability, we have that
\begin{equation*}
\begin{split}
\mathbb{P}[shP(m)=f^m|shP(1)=f^1,...,sh&P(m-1)=f^{m-1}]\\
&=\dfrac{1}{\sum_{i=1}^n(a_{i}+a_{i}^{-1})}\dfrac{Sp_{f^m}^{(n)}(a)}{Sp_{f^{m-1}}^{(n)}(a)}\mathbbm{1}_{f^{m-1}\rightsquigarrow f^m}
\end{split}
\end{equation*}
where the symbol $\mu \rightsquigarrow \lambda$ means that the partition $\lambda$ can be obtained from $\mu$ by either an addition or a deletion of a box.

The last observation summarises to the following result.
\begin{theorem}\label{evolution_YT_Berele}
We fix $n\geq 1$ and $a=(a_{1},...,a_{n})\in \mathbb{R}^n_{>0}$. Let $P(0)$ be the empty tableau, $w_{1},w_{2},...$ be a sequence of independent random variables chosen independently from $\rho=\{\rho(w),w\in [n,\bar{n}]\}$, given in \eqref{eq:letter_distribution}, and $\{P(m), m\geq 1\}$ be the sequence of symplectic tableaux obtained from the empty tableau after sequentially inserting the randomly chosen letters $w_{1},w_{2},...$, i.e. for $m\geq 1$, $P(m)$ is given by
\[P(m)=(P(m-1)\stackrel{\mathcal{B}}{\leftarrow} w_{m}).\]
Then the shape $\{shP(m),m\geq 0\}$ is a Markov chain on the set of partitions $\Lambda_{n}$, started from the empty partition, with transition kernel
\[\Pi(\mu, \lambda) = \dfrac{1}{\sum_{i=1}^n(a_{i}+a_{i}^{-1})}\dfrac{Sp_{\lambda}^{(n)}(a)}{Sp_{\mu}^{(n)}(a)}\mathbbm{1}_{\mu \rightsquigarrow \lambda}.\]
\end{theorem}

Let us consider the $n$-dimensional Markov chain $X = (X_{l}(m),1\leq l \leq n;m\geq 0)$ where for each $1\leq l \leq n$, $X_{l}(0)=0$ and for $m>0$
\[X_{l}(m) = |\{k\leq m:w_{k}=l\}| -|\{k\leq m:w_{k}=\bar{l}\}| \]
where $w_{1},w_{2},...$ are letters from the alphabet $[n,\bar{n}]$ chosen independently with distribution $\rho$. Then $X$ is a $n$-dimensional random walk, started from the origin, in which the $l$-th component jumps to the right with probability $\rho(l)$ and to the left with probability $\rho(\bar{l})$. Let us denote by $\hat{\rho}$ the transition kernel of $X$ killed upon leaving $\Lambda_{n}$. 

From the Pieri identity \eqref{eq:Pieri_Schur} for the symplectic Schur function it is easy to check that the function 
\[h_{n}(x) = \prod_{l=1}^n a_{l}^{-x_{l}}Sp_{x}^{(n)}(a)\]
is harmonic for $\hat{\rho}$. Moreover, we observe that we may re-write the transition kernel $\Pi$ as follows
\[\Pi(\mu, \lambda) = \hat{\rho}(\mu, \lambda)\dfrac{h_{n}(\lambda)}{h_{n}(\mu)}\]
therefore we may identify the shape of a Young tableau obtained from sequentially inserting, randomly chosen with distribution $\rho$, letters from the alphabet $[n,\bar{n}]$ as the Doob $h$-transform of $X$ killed upon leaving $\Lambda_{n}$.
\section{The Berele insertion algorithm for Gelfand-Tsetlin patterns}

As we already mentioned, a different way to represent a symplectic Young tableau is via a symplectic Gelfand-Tsetlin pattern. Let us make this statement more precise. For a symplectic tableau $P$ with entries from $[n,\bar{n}]$ and $k\in [n,\bar{n}]$ let $z^{order(k)}=shP^k$, where $order(k)$ is the order of $k$ in the alphabet $[n,\bar{n}]$ and $P^k$ is the sub-tableau of $P$ that contains only entries of order less or equal to $k$. For example, if 
\begin{figure}[h]
\begin{center}
\scalebox{0.9}{
\begin{tikzpicture}
\node at (-0.5,0) {P = };
\draw [black] (0,0) rectangle (0.8,0.8);
\draw [black] (0.8,0) rectangle (1.6,0.8);
\draw [black] (1.6,0) rectangle (2.4,0.8);
\draw [black] (2.4,0) rectangle (3.2,0.8);
\draw [black] (3.2,0) rectangle (4,0.8);
\draw [black] (0,-0.8) rectangle (0.8,0);
\draw [black] (0.8,-0.8) rectangle (1.6,0);
\node at (0.4,0.4) {1};
\node at (1.2,0.4) {$\bar{1}$};
\node at (2,0.4) {2};
\node at (2.8,0.4) {2};
\node at (3.6,0.4) {$\bar{2}$};
\node at (0.4,-0.4) {$\bar{2}$};
\node at (1.2,-0.4) {$\bar{2}$};
\end{tikzpicture}
}
\end{center} 
\end{figure}

\noindent then $z^1 = shP^{1}=(1)$, $z^{2}=shP^{\bar{1}}=(2)$, $z^3=shP^{2}=(4,0)$ and $z^4 = shP^{\bar{2}}=(5,2)$. By the definition of the symplectic Young tableau it follows that $z^{k-1}\preceq z^k$, for $1\leq k \leq N$. This observation motivates the definition of the symplectic Gelfand-Tsetlin pattern.
\begin{definition}
\label{GT_cone(Berele)}
Let $N$ be a positive integer. A \emph{(symplectic) Gelfand-Tsetlin pattern} $Z = (z^1,...,z^N)$ with $N$ levels is a collection of partitions with $z^{2l-1},z^{2l}\in \Lambda_{l}$, for $1\leq l \leq \big[\frac{N+1}{2}\big]$, which satisfy the interlacing conditions
\[z^1\preceq z^2 \preceq ... \preceq z^N.\]
We denote the set of Gelfand-Tsetlin patterns with $N$ levels by $\mathbb{K}^N$.
\end{definition}
\noindent Schematically, a symplectic Gelfand-Tsetlin pattern is represented as in figure \ref{symp_GT_cone(Berele)}. In this paper we will only focus on patterns with even number of levels $N=2n$.
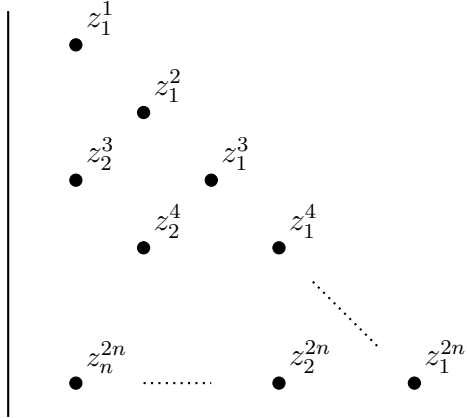
\begin{figure}[h]
\begin{center}
\begin{tikzpicture}[scale = 0.9]

\draw [thick] (0,0) -- (0,6) ;
\draw [fill] (1,5.5) circle [radius=0.09];
\node [above right, black] at (1,5.5) {\large{$z_{1}^1$}};
\draw [fill] (2,4.5) circle [radius=0.09];
\node [above right, black] at (2,4.5) {\large{$z_{1}^2$}};
\draw [fill] (3,3.5) circle [radius=0.09];
\node [above right, black] at (3,3.5) {\large{$z_{1}^3$}};
\draw [fill] (1,3.5) circle [radius=0.09];
\node [above right, black] at (1,3.5) {\large{$z_{2}^3$}};
\draw [fill] (4,2.5) circle [radius=0.09];
\node [above right, black] at (4,2.5) {\large{$z_{1}^4$}};
\draw [fill] (2,2.5) circle [radius=0.09];
\node [above right, black] at (2,2.5) {\large{$z_{2}^4$}};
\draw [fill] (6,0.5) circle [radius=0.09];
\node [above right, black] at (6,.5) {\large{$z_{1}^{2n}$}};
\draw [fill] (4,0.5) circle [radius=0.09];
\node [above right, black] at (4,.5) {\large{$z_{2}^{2n}$}};
\draw [fill] (1,0.5) circle [radius=0.09];
\node [above right, black] at (1,.5) {$z_{n}^{2n}$};
\draw [dotted, thick] (4.5,2) -- (5.5,1);
\draw [dotted, thick] (2,0.5) -- (3,0.5);
\end{tikzpicture}
\end{center} 
\caption{An element of $\mathbb{K}^{2n}$. Both levels with indices $2l$, $2l-1$ contain $l$ particles. At each level the particles satisfy the interlacing property, i.e. $z^{2l-1}_{l}\leq z^{2l}_{l}\leq z^{2l-1}_{l-1}\leq ... \leq z^{2l-1}_{1}\leq z^{2l}_{1}$ and $z^{2l+1}_{l+1}\leq z^{2l}_{l}\leq z^{2l+1}_{l}\leq ... \leq z^{2l}_{1}\leq z^{2l+1}_{1}$.} 
\label{symp_GT_cone(Berele)}
\end{figure}

Before we describe the effect of Berele's insertion algorithm on the particles of a GT pattern, let us \enquote{translate} the changes that may occur to a Young tableau in the language of particles. Suppose that a letter of order $k$ is added at the $i$-th row of a Young tableau, this means that $z^k_{i}$ will increase by $1$. In the language of particles this increase corresponds to a right jump of the particle at position $z^k_{i}$ by one step. Similarly, if a letter of order $k$ is removed from the $i$-th row of a Young tableau then $z^k_{i}$ will decrease by $1$ which represents a left jump for the corresponding particle.

Inserting a letter $l\in [n,\bar{n}]$ leads to an \textbf{attempted} right jump of the particle $z^k_{1}$, where $k=order(l)$ is the order of the letter $l$, at the corresponding Gelfand-Tsetlin pattern.

If a particle $z^k_{i}$, for some $1\leq i \leq [\frac{k+1}{2}]$, $1\leq k \leq 2n$, attempts a right jump, then
\begin{enumerate}[i)]
\item if $z^k_{i}=z^{k+1}_{i}$ the jump is performed and $z^{k+1}_{i}$ is pushed one step to the right. This means that the letter of order $k$ is either added at the end of the $i$-th row or causes a letter of order greater than $k+1$ being bumped from the $i$-th row of the corresponding tableau;
\item if $z^k_{i}<z^{k+1}_{i}$ we have two different cases
\begin{enumerate}[a.]
\item if $i=\frac{k+1}{2}$, with $k$ odd, the jump is suppressed and $z^{k+1}_{i}$ is pulled to the left instead. The transition described here corresponds to the cancellation step in Berele's algorithm;
\item for all the other particles, the jump is performed and the particle $z^{k+1}_{i+1}$ is pulled to the right. This means that a letter of order $k$ is added at the $i$-th row of the corresponding tableau and bumps another letter of order $k+1$ to the row beneath.
\end{enumerate}
\end{enumerate}

If $z^k_{i}$ performs a left jump, then it triggers the leftward move of exactly one of its nearest lower neighbours; the left, $z^{k+1}_{i+1}$, if $z^{k}_{i}=z^{k+1}_{i+1}$ and therefore the jump would lead to violation of the interlacing condition and the right, $z^{k+1}_{i}$, otherwise. The transition we describe here corresponds to the jeu de taquin step of Berele's algorithm. Let us explain why.\\
We assume without loss of generality that $k=2l$ and $z^{2l}_{i}$, for some $1\leq i \leq l$, jumps to the left. This means that a letter $\leq \bar{l}$ is removed from a box in the $i$-th row of the tableau. Note that the empty box can now have only letters $\geq l+1$ to its right. If $z^{2l}_{i}=z^{2l+1}_{i+1}$, then beneath the empty box there is a letter $\leq l+1$, therefore according to the jeu de taquin algorithm we should swap the empty box with the box beneath. Therefore, the $(i+1)$-th row now contains one less box with entry $\leq l+1$, causing $z^{2l+1}_{i+1}$ to jump to the left.\\

Let us close this section by giving an example.
\begin{example}
Suppose $n=2$. Inserting a letter $\bar{1}$, in step a) the particle $z^2_{1}$ performs a rightward jump. Since $z^2_{1}<z^3_{1}$, $z^2_{1}$ triggers the move of $z^3_{2}$. In b) $z^3_{2}$ attempts to jump to the right, but since $z^3_{2}<z^4_{2}$, the jump is suppressed. Finally, in step c) since the jump of the particle $z^3_{2}$ was suppressed, $z^4_{2}$ performs a leftward jump instead.
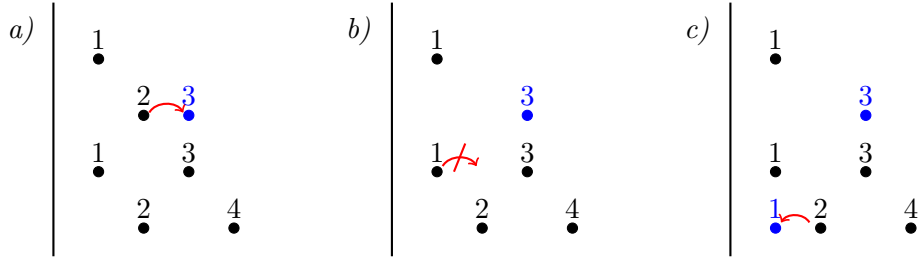
\begin{figure}[h]
\begin{center}
\begin{tikzpicture}[scale = 0.75]
% step 1
\node [right, black] at (-1,6) {\textit{a)}};
\draw [thick] (0,2) -- (0,6.5) ;
\draw [fill] (0.8,5.5) circle [radius=0.09];
\node [above, black] at (0.8,5.5) {1};
\draw [fill] (1.6,4.5) circle [radius=0.09];
\node [above, black] at (1.6,4.5) {2};
\draw [thick, red] [->](1.7,4.55) to [bend left = 60] (2.3,4.55);
\draw [fill, blue] (2.4,4.5) circle [radius=0.09];
\node [above, blue] at (2.4,4.5) {3};
\draw [fill] (2.4,3.5) circle [radius=0.09];
\node [above, black] at (2.4,3.5) {3};
\draw [fill] (0.8,3.5) circle [radius=0.09];
\node [above, black] at (0.8,3.5) {1};
\draw [fill] (3.2,2.5) circle [radius=0.09];
\node [above, black] at (3.2,2.5) {4};
\draw [fill] (1.6,2.5) circle [radius=0.09];
\node [above, black] at (1.6,2.5) {2};

% step 2
\node [right, black] at (5,6) {\textit{b)}};
\draw [thick] (6,2) -- (6,6.5) ;
\draw [fill] (6.8,5.5) circle [radius=0.09];
\node [above, black] at (6.8,5.5) {1};
\draw [fill, blue] (8.4,4.5) circle [radius=0.09];
\node [above, blue] at (8.4,4.5) {3};
\draw [fill] (8.4,3.5) circle [radius=0.09];
\node [above, black] at (8.4,3.5) {3};
\draw [fill] (6.8,3.5) circle [radius=0.09];
\node [above, black] at (6.8,3.5) {1};
\draw [thick, red] [->](6.9,3.6) to [bend left = 60] (7.5,3.6);
\draw [thick, red] (7.1,3.5)--(7.3,4);
\draw [fill] (9.2,2.5) circle [radius=0.09];
\node [above, black] at (9.2,2.5) {4};
\draw [fill] (7.6,2.5) circle [radius=0.09];
\node [above, black] at (7.6,2.5) {2};
% step 3
\node [right, black] at (11,6) {\textit{c)}};
\draw [thick] (12,2) -- (12,6.5) ;
\draw [fill] (12.8,5.5) circle [radius=0.09];
\node [above, black] at (12.8,5.5) {1};
\draw [fill, blue] (14.4,4.5) circle [radius=0.09];
\node [above, blue] at (14.4,4.5) {3};
\draw [fill] (14.4,3.5) circle [radius=0.09];
\node [above, black] at (14.4,3.5) {3};
\draw [fill] (12.8,3.5) circle [radius=0.09];
\node [above, black] at (12.8,3.5) {1};
\draw [fill] (15.2,2.5) circle [radius=0.09];
\node [above, black] at (15.2,2.5) {4};
\draw [fill] (13.6,2.5) circle [radius=0.09];
\node [above, black] at (13.6,2.5) {2};
\draw [fill, blue] (12.8,2.5) circle [radius=0.09];
\node [above, blue] at (12.8,2.5) {1};
\draw [thick, red] [->](13.4,2.6) to [bend right = 60] (12.9,2.6);
\end{tikzpicture}
\end{center} 
\end{figure}
\end{example}

\section{A $q$-deformation of Berele's insertion algorithm}
In this section we consider a generalization of Berele's insertion algorithm that depends on a parameter $q\in (0,1)$. The $q$-deformed algorithm can be thought as a randomization of the usual algorithm, as inserting a letter to a given tableau results to a distribution over a set of tableaux. The proposed $q$-deformation has the property that the shape of the Young tableau again evolves as a Markov chain on $\Lambda_{n}$.

Since a symplectic tableau is equivalent to a symplectic Gelfand-Tsetlin pattern we will present the algorithm in terms of dynamics on the pattern.

For $q\in (0,1)$ and $(x,y) \in \Lambda_{n} \times \Lambda_{n}$ or $(x,y) \in \Lambda_{n-1} \times \Lambda_n$ satisfying the interlacing condition $x\preceq y$, we define the quantities
\begin{equation}
r_{i}(y;x)=q^{y_{i}-x_{i}}\dfrac{1-q^{x_{i-1}-y_{i}}}{1-q^{x_{i-1}-x_{i}}} \text{,\hspace{5pt} }l_{i}(y;x)=q^{x_{i}-y_{i+1}}\dfrac{1-q^{y_{i+1}-x_{i+1}}}{1-q^{x_{i}-x_{i+1}}}
\label{eq:probabilities(Berele)}
\end{equation}
for $1\leq i \leq n$, with the convention that if $x=(x_{1},...,x_{n}) \in \Lambda_{n}$, we set $x_{n+1}\equiv 0$ and $x_{0}\equiv \infty$ (similarly for $y$) and the quantities $l_{i}(y;x),r_{i}(y;x)$ are modified accordingly. A version of the quantities $r_{i},l_{i}$ first appeared in the work of Borodin-Petrov \cite{Borodin_Petrov_2013}.

Inserting a letter $l\in [n,\bar{n}]$ leads to an attempted right jump of the particle $z^k_{1}$, where $k$ denotes the order of the letter $l$, at the corresponding Gelfand-Tsetlin pattern.

If a particle $z^k_{i}$ attempts a rightward jump, then
\begin{enumerate}[i)]
\item if $i=\frac{k+1}{2}$, with $k$ odd, then with probability $r_{i}(z^{k+1};z^k)$, the jump is performed and $z^{k+1}_{i}$ is pushed to the right. Otherwise, with probability $1-r_{i}(z^{k+1};z^k)$ the jump is suppressed and $z^{k+1}_{i}$ is pulled to the left instead;
\item for all the other particles, the jump is performed and either the particle $z^{k+1}_{i}$ is pushed to the right with probability $r_{i}(z^{k+1};z^k)$ or the particle $z^{k+1}_{i+1}$ is pulled to the right with probability $1-r_{i}(z^{k+1};z^k)$.
\end{enumerate}

If $z^k_{i}$ performs a left jump, then it triggers the leftward move of exactly one of its nearest lower neighbours; the left, $z^{k+1}_{i+1}$, with probability $l_{i}(z^{k+1};z^k)$ and the right, $z^{k+1}_{i}$, with probability $1-l_{i}(z^{k+1};z^k)$.\\

By the description of the algorithm, we see that any transition or attempted transition propagates all the way to the end of the pattern. Multiplying all the probabilities obtained till we reach the bottom of the pattern we obtain a probability distribution over the set $\mathbb{K}^{2n}$ of symplectic Gelfand-Tsetlin patterns with $2n$ levels. Let us denote by $I_{l}(Z,\tilde{Z})$ the probability of obtaining a Gelfand-Tsetlin pattern $\tilde{Z}$ for $Z$  after inserting the letter $l\in [n,\bar{n}]$.\\

The $q$-insertion Berele algorithm can be applied to a word $w=w_{1}...w_{m}\in [n,\bar{n}]^m$ starting from a Gelfand-Tsetlin pattern initialised at the origin, i.e. a Gelfand-Tsetlin pattern with all its coordinates equal to zero, which we denote by $\mathbf{0}$. Successively inserting the letters $w_{1},...,w_{m}$ we obtain a distribution of weights $\phi_{w}(\cdot, \cdot)$ on $\mathbb{K}^{2n}\times \mathcal{O}_{m}(n)$ recursively as follows. Set
\begin{equation*}
\phi_{l}(Z,f)= \left \{
\begin{array}{ll}
1 & \text{if } z^k_{i} = \mathbbm{1}_{\{i=1,k\geq \textit{order}(l)\}} \text{ and }f=(\emptyset, (1))\\
0 & \text{otherwise}
\end{array} \right. 
\end{equation*}
where we recall that $\textit{order}(l)$ denotes the order of the letter $l$ in the alphabet $[n,\bar{n}]$.\\
For $w\in [n,\bar{n}]^m$, $l\in [n,\bar{n}]$ and $(\tilde{Z},\tilde{f} = (f^{0},...,f^{m},f^{m+1}))\in \mathbb{K}^{2n}\times \mathcal{O}_{m+1}(n)$ we set
\[\phi_{wl}(\tilde{Z},\tilde{f})=\sum_{Z \in \mathbb{K}^{2n}}\phi_{w}(Z,f)I_{l}(Z, \tilde{Z})\]
where $f = (f^{0},...,f^{m})$.
\section{Main Results}
Before we present the main results, let us introduce some more notation.

The \textit{q-Pochhammer symbol} is written as $(q;q)_{n}$ and defined via the product
\[(q;q)_{n} = \prod_{k=1}^{n}(1-aq^k), \qquad (q;q)_{\infty} = \prod_{k\geq 1} (1-q^k).\]
The \textit{q-factorial} is written as $n!_{q}$ and is defined as 
\[n!_{q} = \dfrac{(q;q)_{n}}{(1-q)^n}.\]
The \textit{q-binomial coefficients} are defined in terms of q-factorials as follows
\[\dbinom{n}{k}_{q} = \dfrac{n!_{q}}{k!_{q}(n-k)!_{q}}.\]
We finally record some properties of the $q$-binomial coefficient that we will use later
\begin{equation}
\begin{split}
\dbinom{n+1}{k}_{q}=\dbinom{n}{k}_{q}\dfrac{1-q^{n+1}}{1-q^{n-k+1}} \qquad \dbinom{n-1}{k}_{q}=\dbinom{n}{k}_{q}\dfrac{1-q^{n-k}}{1-q^n}\\
\dbinom{n}{k+1}_{q}=\dbinom{n}{k}_{q}\dfrac{1-q^{n-k}}{1-q^{k+1}} \qquad \dbinom{n}{k-1}_{q}=\dbinom{n}{k}_{q}\dfrac{1-q^{k}}{1-q^{n-k+1}}
\end{split}
\label{eq:q_binomial}
\end{equation}

\begin{definition}
The \emph{continuous $q$-Hermite polynomials} are defined for $\ell \in \mathbb{Z}_{\geq 0}$ as 
\[H_{\ell}(x|q)=\sum_{m=0}^{\ell} \dbinom{\ell}{m}_{q}e^{i\theta(2m-\ell)}, \quad x = \cos \theta.\]
\end{definition}
\noindent In this paper we will only work with $a = e^{i \theta} \in \mathbb{R}_{>0}$.\\

We now need to define some special Laurent polynomials that generalise the continuous $q$-Hermite polynomials to higher dimensions. Let $n\geq 1$, then for $a=(a_{1},...,a_{n})\in \mathbb{R}^n_{>0}$ and each symplectic Gelfand-Tsetlin pattern $Z \in \mathbb{K}^{2n}$, we define
\[a^Z = \prod_{i=1}^n a_{i}^{2|z^{2i-1}|-|z^{2i}|-|z^{2i-2}|}\]
and
\begin{equation}
\kappa_{n}(Z) = \mathlarger{\prod}_{k=1}^n\mathlarger{‎‎\prod}_{i=1}^{k-1} \dbinom{z^{2k-1}_{i}-z^{2k-1}_{i+1}}{z^{2k-1}_{i}-z^{2k-2}_{i}}_{q} \dbinom{z^{2k}_{i}-z^{2k}_{i+1}}{z^{2k}_{i}-z^{2k-1}_{i}}_{q} \dbinom{z^{2k}_{k}}{z^{2k}_{k}-z^{2k-1}_{k}}_{q}.
\label{eq:q_weights(Berele)}
\end{equation}
For a partition $\lambda = (\lambda_{1}\geq ... \geq \lambda_{n})$ we consider the function
\begin{equation}
P_{\lambda}^{(n)}(a;q) := \sum_{Z\in \mathbb{K}^{2n}[\lambda]}a^Z\kappa_{n}(Z)
\label{eq:symp_q_whittaker(Berele)}
\end{equation}
where $\mathbb{K}^{2n}[\lambda]$ denotes the set of Gelfand-Tsetlin patterns with shape $z^{2n}=\lambda$.\\

When $q = 0$, $\kappa_{n}(Z) = 1$, for every $Z \in \mathbb{K}^{2n}$, therefore 
\[P_{\lambda}^{(n)}(a;q=0)=\sum_{Z \in \mathbb{K}^{2n}[\lambda]}a^Z = Sp_{\lambda}^{(n)}(a).\]
The last equality follows from Definition \ref{Schur_combinatorial} and the mapping between Gelfand-Tsetlin patterns and Young tableaux described in section 4.

For a Gelfand-Tsetlin pattern $Z\in \mathbb{K}^{2n}$, let $Z^{1:2(n-1)}$ be a Gelfand-Tsetlin pattern consisting of the top $2(n-1)$ levels of $Z$. We then note that the weight $\kappa_{n} (Z)$ decomposes as follows
\begin{equation}
\begin{split}
\kappa_{n}(Z) &= \kappa_{n-1}(Z^{1:2(n-1)})\mathlarger{‎‎\prod}_{i=1}^{n-1} \dbinom{z^{2n-1}_{i}-z^{2n-1}_{i+1}}{z^{2n-1}_{i}-z^{2n-2}_{i}}_{q} \dbinom{z^{2n}_{i}-z^{2n}_{i+1}}{z^{2n}_{i}-z^{2n-1}_{i}}_{q} \dbinom{z^{2n}_{n}}{z^{2n}_{n}-z^{2n-1}_{k}}_{q}\\
& = : \kappa_{n-1}(Z^{1:2(n-1)}) \hat{\kappa}_{n-1}^n(z^{2n-2},z^{2n-1},z^{2n}).
\end{split}
\end{equation}
Therefore, the function $P^{(n)}$ exhibits a recursive structure as follows
\[P_{\lambda}^{(n)}(a;q) = \sum a_{n}^{2|z^{2n-1}|-|z^{2n}|-|z^{2n-1}|}\hat{\kappa}_{n-1}^n(z^{2n-2},z^{2n-1},z^{2n}) P^{n-1}_{z^{2n-2}}(\tilde{a};q)\]
where the summation is over $(z^{2n-2},z^{2n-1},z^{2n})\in \Lambda_{n-1}\times \Lambda_{n} \times \Lambda_{n}$ satisfying $z^{2n-2}\preceq z^{2n-1}\preceq z^{2n}\equiv \lambda$ and $\tilde{a}=(a_{1},...,a_{n-1})$ is the vector consisting of the first $n-1$ coordinates of $a$.\\

Let us consider the kernel operator $L_{n}$ defined as
\begin{equation}
L_{n}(\lambda, \mu) = \left\{ \begin{array}{ll}
u_{n,i}^+(\lambda) & \text{ if } \mu=\lambda + e_{i}, 1\leq i \leq n\\
u_{n,i}^-(\lambda) & \text{ if } \mu =\lambda - e_{i}, 1\leq  i \leq n\\
0 & \text{ otherwise}
\end{array} \right.
\label{eq:Pieri_operator}
\end{equation}
where 
\begin{equation}
u_{n,i}^+(\lambda) = \left\{ \begin{array}{ll}
1-q^{\lambda_{i-1}-\lambda_{i}} & \text{ if } 1< i \leq n\\
1 & \text{ if }i=1
\end{array} \right.
\label{eq:Pieri_right}
\end{equation}
\begin{equation}
u_{n,i}^-(\lambda) = \left\{ \begin{array}{ll}
1-q^{\lambda_{i}-\lambda_{i+1}} & \text{ if } 1\leq i < n\\
1-q^{\lambda_{n}} & \text{ if }i=n
\end{array} \right.
\label{eq:Pieri_left}
\end{equation}
We also define the kernel operators $K_{n}$ and $M_{n}$ by
\[K_{n}(\lambda,Z) = a^Z \kappa_{n}(Z)\mathbbm{1}_{z^{2n}=\lambda}, \qquad M_{n}(Z,\tilde{Z}) = \sum_{l\in [n,\bar{n}]}a_{l}I_{l}(Z,\tilde{Z})\]
where $a_{\bar{l}}=a_{l}^{-1}$ for $1\leq l \leq n$.

The operator $L_{n}$ acts on partitions, whereas the operator $M_{n}$ acts on Young tableaux. The kernel $K_{n}$ can be used as a link between a Young tableau and its shape as follows.
\begin{theorem}
\label{intertwining}
The following intertwining relation holds
\[K_{n}M_{n} = L_{n}K_{n}.\]
\end{theorem}

The proof of Theorem \ref{intertwining} can be found in section 8. The above intertwining relation implies that the polynomials $P^{(n)}_{\lambda}(a;q)$ are eigenfunction of the operator $L_{n}$. More specifically the following result holds.
\begin{proposition}\label{eigenrelation}
The following identity holds
\[L_{n}P_{\lambda}^{(n)}(a;q) = \sum_{i=1}^n(a_{i}+a_{i}^{-1})P_{\lambda}^{(n)}(a;q)\]
for all $\lambda \in \Lambda_{n}$ and $a\in \mathbb{R}^n_{>0}$.
\end{proposition}
\begin{proof}
We observe that the function $P^{(n)}_{\lambda}(a;q)$ can be re-written using the kernel $K_{n}$ as follows
\[P_{\lambda}^{(n)}(a;q) =\sum_{Z\in \mathbb{K}^{2n}}a^Z\kappa_{n} (Z)\mathbbm{1}_{z^{2n}=\lambda}= \sum_{Z\in \mathbb{K}^{2n}}K_{n}(\lambda,Z)\]
therefore we have the following
\begin{equation*}
L_{n}P_{\lambda}^{(n)}(a;q) = L_{n} \sum_{Z\in \mathbb{K}^{2n}}K_{n}(\lambda,Z)= \sum_{Z\in \mathbb{K}^{2n}}(L_{n}K_{n})(\lambda,Z)
\end{equation*}
using the intertwining identity from Theorem \ref{intertwining} the last quantity equals
\begin{equation*}
\begin{split}
\sum_{Z\in \mathbb{K}^{2n}}(K_{n}M_{n})(\lambda,Z) & = \sum_{Z\in \mathbb{K}^{2n}} \sum_{\tilde{Z}\in \mathbb{K}^{2n}}K_{n}(\lambda, \tilde{Z})M_{n}(\tilde{Z},Z)\\
&=\sum_{\tilde{Z}\in \mathbb{K}^{2n}}K_{n}(\lambda, \tilde{Z})\sum_{Z\in \mathbb{K}^{2n}} \sum_{l\in [n,\bar{n}]}a_{l}I_{l}(\tilde{Z},Z).
\end{split}
\end{equation*}
We recall that $I_{l}(\tilde{Z},Z)$ denotes the probability of obtaining the Gelfand-Tsetlin pattern $Z$ from $\tilde{Z}$ after inserting the letter $l$, therefore for each $l\in [n,\bar{n}]$ and $\tilde{Z}\in \mathbb{K}^{2n}$, we have
\[\sum_{Z\in \mathbb{K}^{2n}}I_{l}(\tilde{Z},Z)=1\]
and hence we conclude that
\[L_{n}P_{\lambda}^{(n)}(a;q)=\sum_{\tilde{Z}\in \mathbb{K}^{2n}}K_{n}(\lambda, \tilde{Z}) \sum_{l\in [n,\bar{n}]}a_{l}=\sum_{i=1}^n(a_{i}+a_{i}^{-1})P_{\lambda}^{(n)}(a;q)\]
as required.
\end{proof}

The intertwining relation also leads to an expression for the distribution over $\mathbb{K}^{2n}\times \mathcal{O}_{m}(n)$ when we average over all the words of length $m$. 
\begin{theorem}
Let $(Z,f = (f^0,...,f^m))\in \mathbb{K}^{2n}\times \mathcal{O}_{m}(n)$ such that $z^{2n}=f^m$. Then
\begin{equation}
\sum_{w \in [n, \bar{n}]^m}a^{w}\phi_{w}(Z,f) = a^Z\kappa_{n} (Z)\prod_{i=1}^m L_{n}(f^{i-1},f^{i})
\label{eq:weight_simplify}
\end{equation}
where we recall that $a^w = \prod_{i=1}^m a_{w_{i}}$ with $a_{\bar{l}}=a_{l}^{-1}$, for $1\leq l \leq n$.
\end{theorem}
\begin{proof}
Let us denote by $Z(i)$ the Gelfand-Tsetlin pattern obtained after inserting the $i$-th letter to the corresponding Young tableau. For $(Z,f)\in \mathbb{K}^{2n}\times \mathcal{O}_{m}(n)$ with $f = (f^0,...,f^m)$ and $z^{2n}=f^m$ we have
\begin{equation*}
\begin{split}
\sum_{w \in [n,\bar{n}]^m}&a^w\phi_{w}(Z,f)\\
=& \sum_{w\in [n,\bar{n}]^m}a^w \sum_{\substack{(Z(i))_{i=1}^m: \\ z^{2n}(i)=f^i}}I_{w_{1}}(\mathbf{0},Z(1))\, ... \,I_{w_{m}}(Z(m-1),Z)\\
=& \sum_{\substack{(Z(i))_{i=1}^m: \\ z^{2n}(i)=f^i}} \Big(\sum_{w_{1}\in [n,\bar{n}]}a_{w_{1}}I_{w_{1}}(\mathbf{0},Z(1))\Big) \, ... \, \Big(\sum_{w_{m}\in [n,\bar{n}]}a_{w_{m}}I_{w_{m}}(Z(m-1),Z)\Big)\\
=&\sum_{\substack{(Z(i))_{i=1}^m:\\ z^{2n}(i)=f^i}}M_{n}(\mathbf{0},Z(1)) \, ... \, M_{n}(Z(m-1),Z)
\end{split}
\end{equation*}

Since $Z(1)$ is the Gelfand-Tsetlin pattern obtained after inserting a single letter $l\in [n,\bar{n}]$ to an empty tableau, it holds that
\begin{equation*}
z^k_{j}(1) = \left\{ \begin{array}{ll}
1 & \text{ if } j=1,\, k\geq order(l)\\
0 & \text{ otherwise}
\end{array} \right.
\end{equation*}
therefore 
\[M_{n}(\mathbf{0},Z(1)) = a_{l}=K_{n}(f^1,Z(1)).\]
Thus, for the summation over $Z(1)\in \mathbb{K}^{2n}$ such that $z^{2n}(1)=f^1$ we have that 
\begin{equation*}
\begin{split}
\sum_{\substack{Z(1):\\ z^{2n}(1)=f^1}}M_{n}(Z(1),Z(2))M_{n}(\mathbf{0},Z(1))&=\sum_{\substack{Z(1):\\ z^{2n}(1)=f^1}}M_{n}(Z(1),Z(2))K_{n}(f^1,Z(1))\\
&= L_{n}(f^1,f^2)K_{n}(f^2,Z(2))
\end{split}
\end{equation*}
where for the last equality we used the intertwining relation from Theorem \ref{intertwining} along with the fact that inside the summation for $Z(2)$ we have that $z^{2n}(2)=f^2$. \\
Applying repeatedly the intertwining relation we conclude that
\[\sum_{w\in [n,\bar{n}]^m}a^w \phi_{w}(Z,f)=\prod_{i=1}^{m-1}L_{n}(f^i,f^{i+1})K_{n}(f^m,Z)\]
which concludes the result observing that $L_{n}(\emptyset,f^1)=1$.
\end{proof}

Summing \eqref{eq:weight_simplify} over $Z \in \mathbb{K}^{2n}$ and $f \in \mathcal{O}_{m}(n)$ such that $z^{2n}=f^m$ gives the following
\begin{equation*}
\begin{split}
\sum_{f\in \mathcal{O}_{m}(n)}\prod_{i=1}^m L_{n}(f^{i-1},f^i)\sum_{Z\in \mathbb{K}^{2n}[f^m]}a^Z\kappa_{n} (Z) &= \sum_{w\in [n,\bar{n}]^m}a^w \sum_{\substack{(Z,f)\in \mathbb{K}^{2n}\times \mathcal{O}_{m}(n): \\ z^{2n}=f^m}}\phi_{w}(Z,f)\\
\sum_{f\in \mathcal{O}_{m}(n)}\prod_{i=1}^m L_{n}(f^{i-1},f^i) P_{f^m}^{(n)}(a;q) &= \sum_{w\in [n,\bar{n}]^m}a^w 
\end{split}
\end{equation*}
therefore the following result holds.
\begin{corollary}
\label{q-Littlewood(corollary)}
For $a=(a_{1},...,a_{n}) \in \mathbb{R}_{>0}^n$, let $P^{(n)}_{\lambda}(a;q)$ denote the function defined in   \eqref{eq:symp_q_whittaker(Berele)} and 
\begin{equation}
Q^{\lambda}_{m}(n;q) = \sum_{\substack{f\in \mathcal{O}_{m}(n)\\ f^m=\lambda}}\prod_{i=1}^m L_{n}(f^{i-1},f^i)
\label{q_number_of_oscillating}
\end{equation}
then the following Littlewood-type identity holds
\begin{equation}
(a_{1}+a_{1}^{-1}+...+a_{n}+a_{n}^{-1})^m = \sum_{\lambda \in \Lambda_{n}}Q^{\lambda}_{m}(n;q)P_{\lambda}^{(n)}(a;q).
\label{eq:q-Littlewood}
\end{equation}
\end{corollary}
Since $f \in \mathcal{O}_{m}(n)$ we have that for every $1\leq i \leq m$, $f^{i-1}\setminus f^i =(1)$ or $f^{i}\setminus f^{i-1} =(1)$ therefore there exists $1\leq j \leq n$ such that $L_{n}(f^{i-1},f^i) = u_{n,j}^+(f^{i-1})$ or $u_{n,j}^{-}(f^{i-1})$, therefore as $q\to 0$ we have that $L_{n}(f^{i-1},f^i)\to 1$. Hence as $q\to 0$
\[Q^{\lambda}_{m}(n;q)\to \sum_{\substack{f\in \mathcal{O}_{m}(n):\\f^m=\lambda}}1=Q^{\lambda}_{m}(n)\]
where we recall that $Q^{\lambda}_{m}(n)$ denotes the set of oscillating sequences $f=(f^0,...,f^m)\in \mathcal{O}_{m}(n)$ with $f^m=\lambda$. Therefore, the identity \eqref{eq:q-Littlewood} generalises identity \eqref{eq:Schur_decomposition} for $q \in (0,1)$.
%\newpage

\section{A Markov chain on Gelfand-Tsetlin patterns}

In section 5 we proposed a $q$-deformation of the Berele insertion algorithm on Gelfand-Tsetlin patterns. Let us now assume that the inserted letters are chosen randomly. More specifically, for a word $w=w_{1},...,w_{m}$ we assume that $w_{i}$ is chosen independently of the other letters from $[n,\bar{n}]$ with distribution $\rho$, defined in \eqref{eq:letter_distribution}. The probability of obtaining a pair $(Z,f)\in \mathbb{K}^{2n}\times \mathcal{O}_{m}(n)$ is given by
\begin{equation*}
\begin{split}
\mathbb{P}(Z,f)&=\sum_{w\in [n,\bar{n}]^m}\prod_{i=1}^m\rho(w_{i})\phi_{w}(Z,f)\\
&=\sum_{w\in [n,\bar{n}]^m}\dfrac{a^w}{(\sum_{i=1}^n(a_{i}+a_{i}^{-1}))^m}\phi_{w}(Z,f)
\end{split}
\end{equation*}
where $a^w=\prod_{i=1}^m a_{w_{i}}$ with $a_{\bar{l}}=a_{l}^{-1}$.
Using identity \eqref{eq:weight_simplify} we conclude that
\[\mathbb{P}(Z,f)=\dfrac{1}{(\sum_{i=1}^n(a_{i}+a_{i}^{-1}))^m}a^Z\kappa_{n}(Z)\prod_{i=1}^mL_{n}(f^{i-1},f^{i})\]
and summing over $Z \in \mathbb{K}^{2n}$ with $z^{2n}=f^m$ we obtain a distribution over the sequence of shapes for all the intermediate steps
\[\mathbb{P}(f^1,...,f^m)=\dfrac{1}{(\sum_{i=1}^n(a_{i}+a_{i}^{-1}))^m}P_{f^m}^{(n)}(a;q)\prod_{i=1}^mL_{n}(f^{i-1},f^{i}).\]
We moreover have
\begin{equation*}
\mathbb{P}(f^m|f^1,...,f^{m-1})= \dfrac{1}{\sum_{i=1}^n(a_{i}+a_{i}^{-1})} \dfrac{P_{f^m}^{(n)}(a;q)}{P_{f^{m-1}}^{(n)}(a;q)}L_{n}(f^{m-1},f^{m})
\end{equation*}
therefore the sequence of shapes $\{f^m, m\geq 0\}$ evolves as a Markov chain on $\Lambda_{n}$ with transition kernel
\[\Pi(\mu, \lambda)=\dfrac{1}{\sum_{i=1}^n(a_{i}+a_{i}^{-1})} \dfrac{P_{\lambda}^{(n)}(a;q)}{P_{\mu}^{(n)}(a;q)}L_{n}(\mu,\lambda).\]

We summarise the results of this section to the following theorem which generalises Theorem \ref{shape_evolution} for $q\in (0,1)$.
\begin{theorem}\label{shape_evolution}
When applying the $q$-version of the Berele insertion algorithm to a random word $w_{1},w_{2},...$, where each $w_{i}$ is chosen independently at random from $[n,\bar{n}]$ with distribution $\rho$ defined in \eqref{eq:letter_distribution}, the sequence of tableaux $\{Z(m),m\geq 0\}$ obtained evolves as a Markov chain on $\mathbb{K}^{2n}$ with transition kernel
\[M(Z, \tilde{Z})=\sum_{l\in [n,\bar{n}]}\rho(l)I_{l}(Z,\tilde{Z}).\]
The sequence of shapes $f^m=z^{2n}(m)$ evolves as a Markov chain on $\Lambda_{n}$ with transition kernel
\[\Pi(\mu,\lambda)=\dfrac{1}{\sum_{i=1}^n(a_{i}+a_{i}^{-1})} \dfrac{P_{\lambda}^{(n)}(a;q)}{P_{\mu}^{(n)}(a;q)}L_{n}(\mu,\lambda).\]
The conditional law of $Z(m)$, given $\{f^1,...,f^m;f^m=\lambda\}$ is
\[\mathbb{P}(Z(m)=Z|f^1,...,f^m;f^m=\lambda)=\dfrac{K_{n}(\lambda, Z)}{P_{\lambda}^{(n)}(a;q)}.\]
The distribution of $f^m$ is given by
\[\nu(\lambda):=\mathbb{P}(f^m=\lambda)=\dfrac{1}{(\sum_{i=1}^n(a_{i}+a_{i}^{-1}))^m}P^{(n)}_{\lambda}(a;q)Q_{m}^{\lambda}(n;q).\]
\end{theorem}

\begin{remark}
The quantity $\nu(\lambda)$ defines a probability measure on $\Lambda_{n}$, therefore we obtain the following identity for the function $P_{\lambda}^{(n)}(\cdot;q)$
\[\sum_{\lambda \in \Lambda_{n}}P_{\lambda}^{(n)}(a;q)Q_{m}^{\lambda}(n;q) = \Big(\sum_{i=1}^n(a_{i}+a_{i}^{-1})\Big)^m.\]
\end{remark}
 
\section{Proof of intertwining relation}
In this section we will prove the intertwining relation of Theorem \ref{intertwining}. Let us consider the set $T_{n}=\{(x,y,z)\in \Lambda_{n-1}\times \Lambda_{n} \times \Lambda_{n}: x\preceq y\preceq z\}$ and define the kernel $\hat{K}:\Lambda_{n}\times T_{n} \mapsto \mathbb{R}_{\geq 0}$ as follows
\[\hat{K}_{n}(\tilde{z},(x,y,z))=a_{n}^{2|y|-|x|-|z|}\hat{\kappa}_{n}(x,y,z)\mathbbm{1}_{z = \tilde{z}}\]
where 
\[\hat{\kappa}_{n}(x,y,z) = \prod_{i=1}^{n-1}\dbinom{z_{i}-z_{i+1}}{z_{i}-y_{i}}_{q}\dbinom{y_{i}-y_{i+1}}{y_{i}-x_{i}}_{q}\dbinom{z_{n}}{z_{n}-y_{n}}_{q}.\]
We moreover consider the kernel $\hat{M}_{n}:T_{n}\times T_{n}\mapsto \mathbb{R}_{\geq 0}$ given as follows
\newpage
\begin{table}[ht]
\begin{center}
\scalebox{0.88}{
\begin{tabular}{|l|l|}\hline
$(\tilde{x},\tilde{y},\tilde{z})$ & $\hat{M}_{n}((x,y,z),(\tilde{x},\tilde{y},\tilde{z}))$\\ \hline
$(x+e_{i},y+e_{i},z+e_{i})$, $1\leq i \leq n-1$ & $r_{i}(y;x)r_{i}(z;y)u_{n-1,i}^+(x)$\\
$(x+e_{i},y+e_{i},z+e_{i+1})$, $1\leq i \leq n-1$ & $r_{i}(y;x)(1-r_{i}(z;y))u_{n-1,i}^+(x)$\\
$(x+e_{i},y+e_{i+1},z+e_{i+1})$, $1\leq i \leq n-1$ & $(1-r_{i}(y;x))r_{i+1}(z;y)u_{n-1,i}^+(x)$\\
$(x+e_{i},y+e_{i+1},z+e_{i+2})$, $1\leq i \leq n-2$ & $(1-r_{i}(y;x))(1-r_{i+1}(z;y))u_{n-1,i}^+(x)$\\
$(x+e_{n-1},y,z-e_{n})$ & $(1-r_{n-1}(y;x))(1-r_{n}(z;y))u_{n-1,n-1}^+(x)$\\
$(x,y+e_{1},z+e_{1})$ & $a_{n}r_{1}(z;y)$\\
$(x,y+e_{1},z+e_{2})$ & $a_{n}(1-r_{1}(z;y))$\\
$(x,y,z+e_{1})$ & $a_{n}^{-1}$\\
$(x-e_{i},y-e_{i},z-e_{i})$, $1\leq i \leq n-1$ & $(1-l_{i}(y;x))(1-l_{i}(z;y))u_{n-1,i}^-(x)$\\
$(x-e_{i},y-e_{i},z-e_{i+1})$, $1\leq i \leq n-1$ & $(1-l_{i}(y;x))l_{i}(z;y)u_{n-1,i}^-(x)$\\
$(x-e_{i},y-e_{i+1},z-e_{i+1})$, $1\leq i \leq n-2$ & $l_{i}(y;x)(1-l_{i+1}(z;y))u_{n-1,i}^-(x)$\\
$(x-e_{i},y-e_{i+1},z-e_{i+2})$, $1\leq i \leq n-2$ & $l_{i}(y;x)l_{i+1}(z;y)u_{n-1,i}^-(x)$\\
$(x-e_{n-1},y-e_{n},z-e_{n})$ & $l_{n-1}(y;x)u_{n-1,n-1}^-(x)$\\
\hline
\end{tabular}
}
\end{center}
\label{tab:transitions(Berele)}
\end{table}
\noindent where the probabilities $r_{i}$, and $l_{i}$ are defined in \eqref{eq:probabilities(Berele)} and $u_{n-1,i}^{\pm}$ are as in \eqref{eq:Pieri_right} and \eqref{eq:Pieri_left}.
We moreover assume that $\hat{M}_{n}((x,y,z),(\tilde{x},\tilde{y},\tilde{z}))=0$ for any choice of $(\tilde{x},\tilde{y},\tilde{z})$ not listed above.

We will prove that the kernels $\hat{M}_{n}$ and $L_{n}$, defined in \eqref{eq:Pieri_operator}, are intertwined via the kernel $\hat{K}_{n}$. This relation involves only a part of the Gelfand-Tsetlin pattern and hence proving it will be substantially easier to handle than the general intertwining relation.
\begin{proposition}
\label{intertwining_helper}
The following intertwining relation holds
\[\hat{K}_{n} \hat{M}_{n} = L_{n} \hat{K}_{n}.\]
\end{proposition}
\begin{proof}
Since $\hat{K}_{n}$ is supported on $\{z=\tilde{z}\}$ we have that
\begin{equation}
(\hat{K}_{n}\hat{M}_{n})(z,(\tilde{x},\tilde{y},\tilde{z}))=\sum_{\substack{(x,y)\in \Lambda_{n-1}\times \Lambda_{n}:\\ (x,y,z)\in T_{n}}} \hat{K}_{n}(z,(x,y,z)) \hat{M}_{n}((x,y,z),(\tilde{x},\tilde{y},\tilde{z}))
\label{eq:intert_lhs}
\end{equation}
and
\begin{equation}
(L_{n}\hat{K}_{n})(z,(\tilde{x},\tilde{y},\tilde{z}))=L_{n}(z,\tilde{z})\hat{K}_{n}(\tilde{z},(\tilde{x},\tilde{y},\tilde{z})).
\label{eq:intert_rhs}
\end{equation}
We observe that both sides of the intertwining relation vanish unless $\tilde{z}=z \pm e_{i}$ for some $1\leq i \leq n$ hence we only need to confirm the equality of the two sides of the intertwining relation for $z \in \Lambda_{n}$, $(\tilde{x},\tilde{y},\tilde{z})\in T_{n}$ such that $\tilde{z}=z\pm e_{i}$, for $1\leq i \leq n$.

Let us start with the case $\tilde{z}=z+e_{1}$ for which we provide full details of the calculations. The other cases can be checked in a similar manner. In this case the summation in \eqref{eq:intert_lhs} consists of the following terms.\\

\noindent \textbf{I.} If $(x,y,z) = (\tilde{x},\tilde{y},\tilde{z}-e_{1})$ then
\begin{equation*}
\dfrac{\hat{\kappa}_{n}(\tilde{x},\tilde{y},\tilde{z}-e_{1})}{\hat{\kappa}_{n}(\tilde{x},\tilde{y},\tilde{z})} =  \dfrac{\dbinom{(\tilde{z}_{1}-1)-\tilde{z}_{2}}{(\tilde{z}_{1}-1)-\tilde{y}_{1}}_{q}}{\dbinom{\tilde{z}_{1}-\tilde{z}_{2}}{\tilde{z}_{1}-\tilde{y}_{1}}_{q}}=\dfrac{1-q^{\tilde{z}_{1}-\tilde{y}_{1}}}{1-q^{\tilde{z}_{1}-\tilde{z}_{2}}}
\end{equation*}
where we used the property of the $q$-binomial coefficient
\[\binom{n-1}{k-1}_{q}=\dfrac{(q;q)_{n-1}}{(q;q)_{k-1}(q;q)_{n-k}}=\binom{n}{k}_{q}\dfrac{1-q^{k}}{1-q^n}.\]
We then have the following
\begin{equation*}
\begin{split}
\hat{K}_{n}(\tilde{z}-e_{1},(\tilde{x}&,\tilde{y},\tilde{z}-e_{1}))\hat{M}_{n}((\tilde{x},\tilde{y},\tilde{z}-e_{1}),(\tilde{x},\tilde{y},\tilde{z}))\\
=&a_{n}^{2|\tilde{y}|-|\tilde{x}|-|\tilde{z}-e_{1}|}\hat{\kappa}_{n}(\tilde{x},\tilde{y}-e_{1},\tilde{z}-e_{1})a_{n}^{-1}\\
=&a_{n}^{2|\tilde{y}|-|\tilde{x}|-|\tilde{z}|}\hat{\kappa}_{n}(\tilde{x},\tilde{y},\tilde{z})\dfrac{1-q^{\tilde{z}_{1}-\tilde{y}_{1}}}{1-q^{\tilde{z}_{1}-\tilde{z}_{2}}}.
\end{split}
\end{equation*}

\noindent \textbf{II.} If $(x,y,z) = (\tilde{x},\tilde{y}-e_{1},\tilde{z}-e_{1})$ then
\begin{equation*}
\begin{split}
\dfrac{\hat{\kappa}_{n}(\tilde{x},\tilde{y}-e_{1},\tilde{z}-e_{1})}{\hat{\kappa}_{n}(\tilde{x},\tilde{y},\tilde{z})}& =  \dfrac{\dbinom{(\tilde{y}_{1}-1)-\tilde{y}_{2}}{(\tilde{y}_{1}-1)-\tilde{x}_{1}}_{q}}{\dbinom{\tilde{y}_{1}-\tilde{y}_{2}}{\tilde{y}_{1}-\tilde{x}_{1}}_{q}}\dfrac{\dbinom{(\tilde{z}_{1}-1)-\tilde{z}_{2}}{(\tilde{z}_{1}-1)-(\tilde{y}_{1}-1)}_{q}}{\dbinom{\tilde{z}_{1}-\tilde{z}_{2}}{\tilde{z}_{1}-\tilde{y}_{1}}_{q}}\\
&= \dfrac{1-q^{\tilde{y}_{1}-\tilde{x}_{1}}}{1-q^{\tilde{y}_{1}-\tilde{y}_{2}}}\dfrac{1-q^{\tilde{y}_{1}-\tilde{z}_{2}}}{1-q^{\tilde{z}_{1}-\tilde{z}_{2}}}.
\end{split}
\end{equation*}
We then have the following
\begin{equation*}
\begin{split}
\hat{K}_{n}(\tilde{z}-e_{1},(\tilde{x}&,\tilde{y}-e_{1},\tilde{z}-e_{1}))\hat{M}_{n}((\tilde{x},\tilde{y}-e_{1},\tilde{z}-e_{1}),(\tilde{x},\tilde{y},\tilde{z}))\\
=&a_{n}^{2|\tilde{y}-e_{1}|-|\tilde{x}|-|\tilde{z}-e_{1}|}\hat{\kappa}_{n}(\tilde{x},\tilde{y}-e_{1},\tilde{z}-e_{1})a_{n}r_{1}(\tilde{z}-e_{1};\tilde{y}-e_{1})\\
=&a_{n}^{2|\tilde{y}|-|\tilde{x}|-|\tilde{z}|}\hat{\kappa}_{n}(\tilde{x},\tilde{y},\tilde{z})\dfrac{1-q^{\tilde{y}_{1}-\tilde{x}_{1}}}{1-q^{\tilde{y}_{1}-\tilde{y}_{2}}}\dfrac{1-q^{\tilde{y}_{1}-\tilde{z}_{2}}}{1-q^{\tilde{z}_{1}-\tilde{z}_{2}}}q^{\tilde{z}_{1}-\tilde{y}_{1}}.
\end{split}
\end{equation*}

\noindent \textbf{III.} If $(x,y,z) = (\tilde{x}-e_{1},\tilde{y}-e_{1},\tilde{z}-e_{1})$ then
\begin{equation*}
\begin{split}
\dfrac{\hat{\kappa}_{n}(\tilde{x}-e_{1},\tilde{y}-e_{1},\tilde{z}-e_{1})}{\hat{\kappa}_{n}(\tilde{x},\tilde{y},\tilde{z})}& =  \dfrac{\dbinom{(\tilde{y}_{1}-1)-\tilde{y}_{2}}{(\tilde{y}_{1}-1)-(\tilde{x}_{1}-1)}_{q}}{\dbinom{\tilde{y}_{1}-\tilde{y}_{2}}{\tilde{y}_{1}-\tilde{x}_{1}}_{q}}\dfrac{\dbinom{(\tilde{z}_{1}-1)-\tilde{z}_{2}}{(\tilde{z}_{1}-1)-(\tilde{y}_{1}-1)}_{q}}{\dbinom{\tilde{z}_{1}-\tilde{z}_{2}}{\tilde{z}_{1}-\tilde{y}_{1}}_{q}}\\
&= \dfrac{1-q^{\tilde{x}_{1}-\tilde{y}_{2}}}{1-q^{\tilde{y}_{1}-\tilde{y}_{2}}}\dfrac{1-q^{\tilde{y}_{1}-\tilde{z}_{2}}}{1-q^{\tilde{z}_{1}-\tilde{z}_{2}}}.
\end{split}
\end{equation*}
Therefore we conclude that 
\begin{equation*}
\begin{split}
\hat{K}_{n}(\tilde{z}-e_{1},(\tilde{x}-e_{1},\tilde{y}-&e_{1},\tilde{z}-e_{1}))\hat{M}_{n}((\tilde{x}-e_{1},\tilde{y}-e_{1},\tilde{z}-e_{1}),(\tilde{x},\tilde{y},\tilde{z}))\\
=&a_{n}^{2|\tilde{y}-e_{1}|-|\tilde{x}-e_{1}|-|\tilde{z}-e_{1}|}\hat{\kappa}_{n}(\tilde{x}-e_{1},\tilde{y}-e_{1},\tilde{z}-e_{1})\\
 &\times r_{1}(\tilde{y}-e_{1};\tilde{x}-e_{1})r_{1}(\tilde{z}-e_{1};\tilde{y}-e_{1})u_{n,1}^+(\tilde{x}-e_{1})\\
=&a_{n}^{2|\tilde{y}|-|\tilde{x}|-|\tilde{z}|}\hat{\kappa}_{n}(\tilde{x},\tilde{y},\tilde{z})\dfrac{1-q^{\tilde{x}_{1}-\tilde{y}_{2}}}{1-q^{\tilde{y}_{1}-\tilde{y}_{2}}}\dfrac{1-q^{\tilde{y}_{1}-\tilde{z}_{2}}}{1-q^{\tilde{z}_{1}-\tilde{z}_{2}}}q^{\tilde{z}_{1}-\tilde{x}_{1}}.
\end{split}
\end{equation*}
Adding the three terms together we conclude that
\[(\hat{K}_{n}\hat{M}_{n})(\tilde{z}-e_{1},(\tilde{x},\tilde{y},\tilde{z}))=\hat{K}_{n}(\tilde{z},(\tilde{x},\tilde{y},\tilde{z}))\]
which, since $L_{n}(\tilde{z}-e_{1},\tilde{z})=u_{n,1}^+(\tilde{z}-e_{1})=1$, equals $(L_{n}\hat{K}_{n})(\tilde{z}-e_{1},(\tilde{x},\tilde{y},\tilde{z}))$. Therefore for this case the two sides of the intertwining relation are equal.\\

If $\tilde{z} = z-e_{1}$ then the summation in \eqref{eq:intert_lhs} consists of a single term corresponding to $(x,y,z)=(\tilde{x}+e_{1},\tilde{y}+e_{1},\tilde{z}+e_{1})$ and equals
\begin{equation*}
\begin{split}
\hat{K}_{n}(\tilde{z}+e_{1},&(\tilde{x}+e_{1},\tilde{y}+e_{1},\tilde{z}+e_{1}))\hat{M}_{n}((\tilde{x}+e_{1},\tilde{y}+e_{1},\tilde{z}+e_{1}),(\tilde{x},\tilde{y},\tilde{z}))\\
=&a_{n}^{2|\tilde{y}+e_{1}|-|\tilde{x}+e_{1}|-|\tilde{z}+e_{1}|}\hat{\kappa}_{n}(\tilde{x}+e_{1},\tilde{y}+e_{1},\tilde{z}+e_{1})\\
 &\quad \times (1-l_{1}(\tilde{y}+e_{1};\tilde{x}+e_{1}))(1-l_{1}(\tilde{z}+e_{1};\tilde{y}+e_{1}))u_{n,1}^-(\tilde{x}+e_{1})\\
=&a_{n}^{2|\tilde{y}|-|\tilde{x}|-|\tilde{z}|}\hat{\kappa}_{n}(\tilde{x},\tilde{y},\tilde{z})(1-q^{\tilde{z}_{1}-\tilde{z}_{2}+1}).
\end{split}
\end{equation*}
The last quantity equals $u_{n,1}^-(\tilde{z}+e_{1})\hat{K}_{n}(\tilde{z},(\tilde{x},\tilde{y},\tilde{z}))$ which gives \eqref{eq:intert_rhs} for $\tilde{z}=z-e_{1}$, as required.\\

If $\tilde{z}=z+e_{i}$, with $2\leq i \leq n$, the right-hand side of \eqref{eq:intert_lhs} of terms corresponding to the following choices for $(x,y,z)$
\begin{enumerate}[i)]
\item 
	\begin{enumerate}[a.]
	\item$(x,y,z)=(\tilde{x}-e_{i-2},\tilde{y}-e_{i-1},\tilde{z}-e_{i})$ if $i>2$,
	\item $(x,y,z)=(\tilde{x},\tilde{y}-e_{1},\tilde{z}-e_{2})$ if $i=2$;
	\end{enumerate}
\item $(x,y,z)=(\tilde{x}-e_{i-1},\tilde{y}-e_{i-1},\tilde{z}-e_{i})$;
\item $(x,y,z)=(\tilde{x}-e_{i-1},\tilde{y}-e_{i},\tilde{z}-e_{i})$;
\item $(x,y,z)=(\tilde{x}-e_{i},\tilde{y}-e_{i},\tilde{z}-e_{i})$ for $i<n$.
\end{enumerate}
The corresponding terms are
\begin{enumerate}[i)]
\item $a_{n}^{2|\tilde{y}|-|\tilde{x}|-|\tilde{z}|}\hat{\kappa}_{n}(\tilde{x},\tilde{y},\tilde{z})(1-q^{\tilde{z}_{i-1}-\tilde{z}_{i}+1})\dfrac{(1-q^{\tilde{y}_{i-1}-\tilde{x}_{i-1}})(1-q^{\tilde{z}_{i}-\tilde{y}_{i}})}{(1-q^{\tilde{z}_{i}-\tilde{z}_{i+1}})(1-q^{\tilde{y}_{i-1}-\tilde{y}_{i}})}$;
\item $a_{n}^{2|\tilde{y}|-|\tilde{x}|-|\tilde{z}|}\hat{\kappa}_{n}(\tilde{x},\tilde{y},\tilde{z})(1-q^{\tilde{z}_{i-1}-\tilde{z}_{i}+1})\dfrac{q^{\tilde{y}_{i-1}-\tilde{x}_{i-1}}(1-q^{\tilde{x}_{i-1}-\tilde{y}_{i}})(1-q^{\tilde{z}_{i}-\tilde{y}_{i}})}{(1-q^{\tilde{z}_{i}-\tilde{z}_{i+1}})(1-q^{\tilde{y}_{i-1}-\tilde{y}_{i}})}$;
\item $a_{n}^{2|\tilde{y}|-|\tilde{x}|-|\tilde{z}|}\hat{\kappa}_{n}(\tilde{x},\tilde{y},\tilde{z})(1-q^{\tilde{z}_{i-1}-\tilde{z}_{i}+1})\dfrac{q^{\tilde{z}_{i}-\tilde{y}_{i}}(1-q^{\tilde{y}_{i}-\tilde{x}_{i}})(1-q^{\tilde{y}_{i}-\tilde{z}_{i+1}})}{(1-q^{\tilde{z}_{i}-\tilde{z}_{i+1}})(1-q^{\tilde{y}_{i}-\tilde{y}_{i+1}})}$;
\item $a_{n}^{2|\tilde{y}|-|\tilde{x}|-|\tilde{z}|}\hat{\kappa}_{n}(\tilde{x},\tilde{y},\tilde{z})(1-q^{\tilde{z}_{i-1}-\tilde{z}_{i}+1})\dfrac{q^{\tilde{z}_{i}-\tilde{x}_{i}}(1-q^{\tilde{y}_{i}-\tilde{z}_{i+1}})(1-q^{\tilde{x}_{i}-\tilde{y}_{i+1}})}{(1-q^{\tilde{z}_{i}-\tilde{z}_{i+1}})(1-q^{\tilde{y}_{i}-\tilde{y}_{i+1}})}$.
\end{enumerate}
Gathering all the terms together we conclude that if $\tilde{z}=z+e_{i}$, then \eqref{eq:intert_lhs} equals $(1-q^{\tilde{z}_{i-1}-\tilde{z}_{i}+1})\hat{K}_{n}(\tilde{z},(\tilde{x},\tilde{y},\tilde{z})) = (L_{n}\hat{K}_{n})(\tilde{z}-e_{i},(\tilde{x},\tilde{y},\tilde{z}))$.\\

Finally, we conclude the proof of the Proposition confirming the intertwining relation for $\tilde{z}=z-e_{i}$, for some $2\leq i \leq n$. In this case, we have the contribution of four terms corresponding to the following 
\begin{enumerate}[i)]
\item $(x,y,z)=(\tilde{x}+e_{i-2},\tilde{y}+e_{i-1},\tilde{z}+e_{i})$, for $i>2$;
\item $(x,y,z)=(\tilde{x}+e_{i-1},\tilde{y}+e_{i-1},\tilde{z}+e_{i})$;
\item $(x,y,z)=(\tilde{x}+e_{i-1},\tilde{y}+e_{i},\tilde{z}+e_{i})$;
\item \begin{enumerate}[a)] \item $(x,y,z)=(\tilde{x}+e_{i},\tilde{y}+e_{i},\tilde{z}+e_{i})$, if $i<n$ \item $(x,y,z)=(\tilde{x}-e_{n-1},\tilde{y},\tilde{z}+e_{n})$, if $i=n$.\end{enumerate}
\end{enumerate}
The corresponding terms are given by
\begin{enumerate}[i)]
\item $a_{n}^{2|\tilde{y}|-|\tilde{x}|-|\tilde{z}|}\hat{\kappa}_{n}(\tilde{x},\tilde{y},\tilde{z})(1-q^{\tilde{z}_{i}-\tilde{z}_{i+1}+1})\dfrac{q^{\tilde{x}_{i-2}-\tilde{z}_{i}}(1-q^{y_{i-2}'-\tilde{x}_{i-2}})(1-q^{\tilde{z}_{i-1}-\tilde{y}_{i-1}})}{(1-q^{\tilde{z}_{i-1}-\tilde{z}_{i}})(1-q^{y_{i-2}'-\tilde{y}_{i-1}})}$;
\item $a_{n}^{2|\tilde{y}|-|\tilde{x}|-|\tilde{z}|}\hat{\kappa}_{n}(\tilde{x},\tilde{y},\tilde{z})(1-q^{\tilde{z}_{i}-\tilde{z}_{i+1}+1})\dfrac{q^{\tilde{y}_{i-1}-\tilde{z}_{i}}(1-q^{\tilde{x}_{i-2}-\tilde{y}_{i-1}})(1-q^{\tilde{z}_{i-1}-\tilde{y}_{i-1}})}{(1-q^{\tilde{z}_{i-1}-\tilde{z}_{i}})(1-q^{y_{i-2}'-\tilde{y}_{i-1}})}$;
\item $a_{n}^{2|\tilde{y}|-|\tilde{x}|-|\tilde{z}|}\hat{\kappa}_{n}(\tilde{x},\tilde{y},\tilde{z})(1-q^{\tilde{z}_{i}-\tilde{z}_{i+1}+1})\dfrac{q^{\tilde{x}_{i-1}-\tilde{y}_{i}}(1-q^{\tilde{y}_{i-1}-\tilde{x}_{i-1}})(1-q^{\tilde{y}_{i-1}-\tilde{z}_{i}})}{(1-q^{\tilde{z}_{i-1}-\tilde{z}_{i}})(1-q^{\tilde{y}_{i-1}-\tilde{y}_{i}})}$;
\item $a_{n}^{2|\tilde{y}|-|\tilde{x}|-|\tilde{z}|}\hat{\kappa}_{n}(\tilde{x},\tilde{y},\tilde{z})(1-q^{\tilde{z}_{i}-\tilde{z}_{i+1}+1})\dfrac{(1-q^{\tilde{y}_{i-1}-\tilde{z}_{i}})(1-q^{\tilde{x}_{i-1}-\tilde{y}_{i}})}{(1-q^{\tilde{z}_{i-1}-\tilde{z}_{i}})(1-q^{\tilde{y}_{i-1}-\tilde{y}_{i}})}$.
\end{enumerate}
Adding the four terms together we conclude that if $\tilde{z}=z-e_{i}$, for $2\leq i \leq n$, then the right-hand side of \eqref{eq:intert_lhs} equals $(1-q^{\tilde{z}_{i}-\tilde{z}_{i+1}+1})\hat{K}_{n}(\tilde{z},(\tilde{x},\tilde{y},\tilde{z}))=(L_{n}\hat{K}_{n})(\tilde{z}+e_{i},(\tilde{x},\tilde{y},\tilde{z}))$.
\end{proof}

Let us now proceed to the proof of Theorem \ref{intertwining}. We will prove the result by induction on $n$. \\

\noindent \textbf{\underline{Base Case:}} We start by establishing the result for $n=1$. For two Gelfand-Tsetlin patterns $Z,\tilde{Z}\in \mathbb{K}^2$ let us write $z^1_{1}=x$, $z^{2}_{1}=y$ and $\tilde{z}^1_{1}=\tilde{x}$, $\tilde{z}^2_{1}=\tilde{y}$. Then the kernels $K_{1}$ and $M_{1}$ are as follows
\[K_{1}(\lambda, Z)=a_{1}^{2x-y}\dbinom{y}{y-x}_{q}\mathbbm{1}_{y=\lambda}\]
and
\begin{equation*}
M_{1}((x,y),(\tilde{x},\tilde{y})) = \left\{ \begin{array}{ll}
a_{1} r_{1}(y;x) & \text{ if }(\tilde{x},\tilde{y})=(x+1,y+1)\\
a_{1}(1- r_{1}(y;x)) & \text{ if }(\tilde{x},\tilde{y})=(x,y-1)\\
a_{1}^{-1} & \text{ if }(\tilde{x},\tilde{y})=(x,y+1)\\
0 & \text{ otherwise}
\end{array} \right. .
\end{equation*}
We then need to prove, for every $y\in \mathbb{Z}_{\geq 0}$ and $(\tilde{x},\tilde{y}) \in T_{1}:= \{\tilde{x},\tilde{y}\in \mathbb{Z}_{\geq 0}: \tilde{x}\leq \tilde{y}\}$, the following identity
\[(K_{1}M_{1})(y,(\tilde{x},\tilde{y}))=(L_{1}K_{1})(\tilde{x},\tilde{y})\]
which simplifies to
\begin{equation}
\sum_{\substack{x \in \mathbb{Z}_{\geq 0}:\\ (x,y)\in T_{1}}}K_{1}(y,(x,y))M_{1}((x,y),(\tilde{x},\tilde{y})) = L_{1}(y,\tilde{y})K_{1}(\tilde{y},(\tilde{x},\tilde{y}))
\label{eq:intert_base}
\end{equation}

We observe that both sides of \eqref{eq:intert_base} are equal to zero, unless $y = \tilde{y}\pm 1$, therefore we need to confirm the intertwining for every $(\tilde{x},\tilde{y}) \in T_{1}$ and $y = \tilde{y}\pm 1$.\\

Starting with the case $y = \tilde{y}- 1$ we have two non-zero terms. Using the properties of the $q$-binomial coefficient, we recorded in \eqref{eq:q_binomial}, we calculate the contribution of each term to the sum at the left-hand side of \eqref{eq:intert_base}. \\
\textbf{I.} if $(x,y) = (\tilde{x},\tilde{y}-1)$ then
\begin{equation*}
\begin{split}
K_{1}(\tilde{y}-1,(\tilde{x},\tilde{y}-1))M_{1}(&(\tilde{x},\tilde{y}-1),(\tilde{x},\tilde{y}))\\
&=a_{1}^{2\tilde{x}-\tilde{y}+1}\dbinom{\tilde{y}-1}{\tilde{y}-\tilde{x}-1}_{q}a_{1}^{-1}\\
&=a_{1}^{2\tilde{x}-\tilde{y}}\dbinom{\tilde{y}}{\tilde{y}-\tilde{x}}_{q}\dfrac{1-q^{\tilde{y}-\tilde{x}}}{1-q^{\tilde{y}}}\\
&=K_{1}(\tilde{y},(\tilde{x},\tilde{y}))\dfrac{1-q^{\tilde{y}-\tilde{x}}}{1-q^{\tilde{y}}}.
\end{split}
\end{equation*}
\textbf{II.} if $(x,y) = (\tilde{x}-1,\tilde{y}-1)$ we have
\begin{equation*}
\begin{split}
K_{1}(\tilde{y}-1,(\tilde{x}-1,\tilde{y}-1)&)M_{1}((\tilde{x}-1,\tilde{y}-1),(\tilde{x},\tilde{y}))\\
&=a_{1}^{2\tilde{x}-\tilde{y}-1}\dbinom{\tilde{y}-1}{\tilde{y}-\tilde{x}}_{q}a_{1}(1-r_{1}(\tilde{y}-1;\tilde{x}-1))\\
&=a_{1}^{2\tilde{x}-\tilde{y}}\dbinom{\tilde{y}}{\tilde{y}-\tilde{x}}_{q}\dfrac{(1-q^{\tilde{x}})q^{\tilde{y}-\tilde{x}}}{1-q^{\tilde{y}}}\\
&=K_{1}(\tilde{y},(\tilde{x},\tilde{y}))\dfrac{(1-q^{\tilde{x}})q^{\tilde{y}-\tilde{x}}}{1-q^{\tilde{y}}}.
\end{split}
\end{equation*}
These two cases together sum up to $K_{1}(\tilde{y},(\tilde{x},\tilde{y}))$ which, since $L_{1}(\tilde{y}-1,\tilde{y})=u_{1,1}^+(\tilde{y}-1) = 1$, equals the right-hand side of \eqref{eq:intert_base}.\\

For the case $y=\tilde{y}+1$. The left-hand side of \eqref{eq:intert_base} involves a single term corresponding to $(x,y)= (\tilde{x},\tilde{y}+1)$ which equals
\[a_{1}^{2\tilde{x}-\tilde{y}-1}\dbinom{\tilde{y}+1}{\tilde{y}-\tilde{x}+1}_{q}a_{1}(1-q^{\tilde{y}-\tilde{x}+1})=K_{1}(\tilde{y},(\tilde{x},\tilde{y}))(1-q^{\tilde{y}+1})\]
where the last quantity gives the right-hand side of \eqref{eq:intert_base} for $y= \tilde{y}+1$.\\

\noindent \textbf{\underline{General Case:}} 
In order to avoid heavy notation we will write $(x,y,z)$ instead of $(z^{2n-2},z^{2n-1},z^{2n})$ for the bottom three levels of a Gelfand-Tsetlin pattern $Z\in \mathbb{K}^{2n}$.\\
Let us assume that the intertwining relation holds for $n-1$, i.e.
\[K_{n-1}M_{n-1}=L_{n-1}K_{n-1}.\]
We observe that the kernel $K_{n}$ can be decomposed as follows
\[K_{n}(\lambda,Z)=K_{n-1}(x,Z^{1:2(n-1)})\hat{K}_{n}(\lambda,(x,y,z)).\] 
Regarding the kernel $M_{n}$ we have
\[M_{n}(Z,\tilde{Z})=\big(\sum_{l\in [n,\bar{n}]\setminus \{n,\bar{n}\} }+\sum_{l\in \{n,\bar{n}\}}\big)a_{l}I_{l}(Z,\tilde{Z}).\]
Inserting a letter $n$ or $\bar{n}$ affects only the bottom two levels of the Gelfand-Tsetlin pattern and the second summation equals
\begin{equation*}
\left\{ \begin{array}{ll}
a_{n} r_{1}(z;y) & \text{ if }(\tilde{y},\tilde{z})=(y+e_{1},z+e_{1})\\
a_{n}(1- r_{1}(z;y)) & \text{ if }(\tilde{y},\tilde{z})=(y+e_{1},z+e_{2})\\
a_{n}^{-1} & \text{ if }(\tilde{y},\tilde{z})=(y,z-e_{1})
\end{array} \right. 
\end{equation*}
which is equal to 
\[\hat{M}_{n}((x,y,z),(\tilde{x},\tilde{y},\tilde{z}))\mathbbm{1}_{\tilde{x}=x}.\]

Let us now denote by $\mathbb{P}_{x\to \tilde{x}}\big((\tilde{y},\tilde{z})|(y,z)\big)$ the probability that the bottom two levels of the Gelfand-Tsetlin pattern performed the transition $(y,z)\to (\tilde{y},\tilde{z})$ as a result of the transition of the particles at level $2n-2$ from $x$ to $ \tilde{x}$. It is easy to see that if $\tilde{x}\neq x$ then
\begin{equation*}
\mathbb{P}_{x\to \tilde{x}}\big((\tilde{y},\tilde{z})|(y,z)\big)= \dfrac{\hat{M}_{n}((x,y,z),(\tilde{x},\tilde{y},\tilde{z}))}{L_{n-1}(x,\tilde{x})}.
\end{equation*}
Therefore, if $l \in [n,\bar{n}]\setminus \{n,\bar{n}\}$ then $I_{l}(Z,\bar{Z})$ can be decomposed as follows
\[I_{l}(Z,\bar{Z})=I_{l}(Z^{1:2(n-1)},\tilde{Z}^{1:2(n-1)})\dfrac{\hat{M}_{n}((x,y,z),(\tilde{x},\tilde{y},\tilde{z}))}{L_{n-1}(x,\tilde{x})}\]
and the first summation equals
\begin{equation*}
\sum_{l \in [n,\bar{n}]\setminus \{n,\bar{n}\}}a_{l}I_{l}(Z,\tilde{Z})=M_{n-1}(Z^{1:2(n-1)},\tilde{Z}^{1:2(n-1)}) \dfrac{\hat{M}_{n}((x,y,z),(\tilde{x},\tilde{y},\tilde{z}))}{L_{n-1}(x,\tilde{x})}\mathbbm{1}_{\tilde{x}\neq x}.
\end{equation*}\\
For $\lambda \in \Lambda_{n}$ and $\tilde{Z}\in \mathbb{K}^{2n}$ we have
\[(K_{n}M_{n})(\lambda, \tilde{Z})= \big(\sum_{\substack{Z \in \mathbb{K}^{2n}:\\ x=\tilde{x}}} + \sum_{\substack{Z \in \mathbb{K}^{2n}:\\ x\neq\tilde{x}}} \big)K_{n}(\lambda,Z)M_{n}(Z,\tilde{Z}):= I_{1}+I_{2}.\]
If $x=\tilde{x}$ then this implies that $Z^{1:2(n-1)}=\tilde{Z}^{1:2(n-1)}$ therefore the first summation equals
\begin{equation*}
I_{1}=\sum_{\substack{(x,y,z)\in T_{n}:\\ x=\tilde{x}}}K_{n-1}(\tilde{x},\tilde{Z}^{1:2(n-1)})\hat{K}_{n}(\lambda,(x,y,z)) \hat{M}_{n}((x,y,z),(\tilde{x},\tilde{y},\tilde{z})).
\end{equation*}
For the second summation we have
\begin{equation*}
\begin{split}
I_{2}&=\sum_{\substack{Z \in \mathbb{K}^{2n}:\\ x\neq \tilde{x}}}\dfrac{\hat{M}_{n}((x,y,z),(\tilde{x},\tilde{y},\tilde{z}))}{L_{n-1}(x,\tilde{x})}\hat{K}_{n}(\lambda,(x,y,z))\\
& \hspace{80pt} \times K_{n-1}(x,Z^{1:2(n-1)}) M_{n-1}(Z^{1:2(n-1)},\tilde{Z}^{1:2(n-1)})\\
& \\
&= \sum_{\substack{(x,y,z)\in T_{n}:\\ x\neq \tilde{x}}}\dfrac{\hat{M}_{n}((x,y,z),(\tilde{x},\tilde{y},\tilde{z}))}{L_{n-1}(x,\tilde{x})}\hat{K}_{n}(\lambda,(x,y,z))\\
& \hspace{80pt} \times (K_{n-1}M_{n-1})(x,\tilde{Z}^{1:2(n-1)})
\end{split}
\end{equation*}
using the induction hypothesis we then conclude that
\begin{equation*}
I_{2}=\sum_{\substack{(x,y,z)\in T_{n}:\\ x\neq \tilde{x}}}K_{n-1}(\tilde{x},\tilde{Z}^{1:2(n-1)})\hat{K}_{n}(\lambda,(x,y,z)) \hat{M}_{n}((x,y,z),(\tilde{x},\tilde{y},\tilde{z})).
\end{equation*}
Combining the two sums $I_{1}$ and $I_{2}$ leads to the following
\begin{equation*}
(K_{n}M_{n})(\lambda, \tilde{Z})=K_{n-1}(\tilde{x},\tilde{Z}^{1:2(n-1)})(\hat{K}_{n}\hat{M}_{n})(\lambda,(\tilde{x},\tilde{y},\tilde{z})).
\end{equation*}
Finally, using Proposition \ref{intertwining_helper} we conclude that
\begin{equation*}
\begin{split}
(K_{n}M_{n})(\lambda, \tilde{Z})&=K_{n-1}(\tilde{x},\tilde{Z}^{1:2(n-1)})(L_{n}\hat{K}_{n})(\lambda,(\tilde{x},\tilde{y},\tilde{z}))\\
&=\sum_{\mu \in \Lambda_{n}} L_{n}(\lambda,\mu)\hat{K}_n(\mu,(\tilde{x},\tilde{y},\tilde{z}))K_{n-1}(\tilde{x},\tilde{Z}^{1:2(n-1)})\\
&=\sum_{\mu \in \Lambda_{n}} L_{n}(\lambda,\mu)K_{n}(\mu, \tilde{Z})\\
&= (L_{n}K_{n})(\lambda, \tilde{Z})
\end{split}
\end{equation*}
as required.

\section*{Acknowledgements} The author would like to thank Nikos Zygouras and Jon Warren for their guidance. Thanks are also due to Bruce Westbury for mentioning the Berele insertion algorithm which motivated the work of this paper. This material is part of the author's PhD thesis, supported by departmental bursary from the Department of Statistics of University of Warwick.
\bibliographystyle{eg-alpha}
\bibliography{refs}

\end{document}